\documentclass[10pt]{amsart}
\usepackage{amsmath}
\usepackage{amsthm}
\usepackage[alphabetic,initials]{amsrefs}
\usepackage[all]{xy}
\usepackage{tikz}
\usepackage{tikz-cd}
\usetikzlibrary{%
  matrix,%
  calc,%
  arrows%
}
\usepackage[parfill]{parskip}    
\usepackage{graphicx}
\usepackage{amssymb}
\usepackage{epstopdf}
\usepackage[utf8]{inputenc}
\DeclareMathOperator{\coker}{coker}

\newcommand{\SDR}[5]{\xymatrix{*[r]{#1} \ar@<1ex>[r]^-{#3} \ar@(ul,dl)[]_{#5} & #2 \ar@<1ex>[l]^-{#4}}}
\newcommand{\bigSDR}[5]{\xymatrix{*[r]{#1} \ar@<1ex>[rr]^-{#3} \ar@(ul,dl)[]_{#5} && #2 \ar@<1ex>[ll]^-{#4}}}
\newcommand{\bigbigSDR}[5]{\xymatrix{*[r]{#1} \ar@<1ex>[rrr]^-{#3} \ar@(ul,dl)[]_{#5} &&& #2 \ar@<1ex>[lll]^-{#4}}}
\newcommand{\kk}{\Bbbk}
\newcommand{\tensor}{\otimes}
\newcommand{\CC}{\mathbb{C}}
\newcommand{\QQ}{\mathbb{Q}}
\newcommand{\RR}{\mathbb{R}}
\newcommand{\ZZ}{\mathbb{Z}}
\newcommand{\FF}{\mathbb{F}}
\newcommand{\Der}{\operatorname{Der}}
\newcommand{\Coder}{\operatorname{Coder}}
\newcommand{\Hom}{\operatorname{Hom}}
\newcommand{\Ext}{\operatorname{Ext}}
\newcommand{\Tw}{\operatorname{Tw}}
\newcommand{\ad}{\operatorname{ad}}
\newcommand{\loopr}{\bullet}
\newcommand{\as}{\text{\normalfont{<}}} 

\newcommand{\tDer}{\widetilde{\operatorname{\Der}}}
\newcommand{\tCoder}{\widetilde{\operatorname{\Coder}}}
\newcommand{\UU}{\mathbb{U}}
\newcommand{\LL}{\mathbb{L}}

\newcommand{\CP}[1]{\CC \hspace{-2.5pt} \operatorname{P}^{#1}}

\newcommand{\rtree}{\xygraph{
!{<0pt,0pt>;<11pt,0pt>:<0pt,11pt>::}
!{(-4,3)}*{{}_g}="a"
!{(0,3)}*{{}_g}="b"
!{(4,3)}*{{}_g}="c"
!{(-2,1)}*+{{}_{m_2}}="d"
!{(0,-1)}*+{{}_{m_2}}="e"
!{(0,-3)}*+{{}_f}="f"
"a"-"d"
"b"-"d"
"c"-"e"
"d"-"e"|{h}
"e"-"f"
}}

\newcommand{\ltree}{\xygraph{
!{<0pt,0pt>;<-11pt,0pt>:<0pt,11pt>::}
!{(-4,3)}*{{}_g}="a"
!{(0,3)}*{{}_g}="b"
!{(4,3)}*{{}_g}="c"
!{(-2,1)}*+{{}_{m_2}}="d"
!{(0,-1)}*+{{}_{m_2}}="e"
!{(0,-3)}*+{{}_f}="f"
"a"-"d"
"b"-"d"
"c"-"e"
"d"-"e"|{h}
"e"-"f"
}}

\theoremstyle{plain}
\newtheorem{theorem}{Theorem}[section]
\newtheorem{proposition}[theorem]{Proposition}
\newtheorem{lemma}[theorem]{Lemma}
\newtheorem{corollary}[theorem]{Corollary}

\theoremstyle{definition}
\newtheorem{definition}[theorem]{Definition}
\newtheorem{example}[theorem]{Example}
\newtheorem{observation}[theorem]{Observation}

\theoremstyle{remark}
\newtheorem{remark}[theorem]{Remark}

\title{Free loop space homology of highly connected manifolds}
\author{Alexander Berglund and Kaj B\"orjeson} 
\address{Department of Mathematics, Stockholm University, SE-106 91 Stockholm, Sweden}
\email{alexb@math.su.se, kaj@math.su.se}

\begin{document}

\begin{abstract}
We calculate the homology of the free loop space of $(n-1)$-connected closed manifolds of dimension at most $3n-2$ ($n\geq 2$), with the Chas-Sullivan loop product and loop bracket. Over a field of characteristic zero, we obtain an expression for the BV-operator. We also give explicit formulas for the Betti numbers, showing they grow exponentially. Our main tool is the connection between formality, coformality and Koszul algebras that was elucidated by the first author \cite{Berglund}.
\end{abstract}

\maketitle

\section{Introduction}
In \cite{ChasSullivan}, Chas and Sullivan defined a loop product and a loop bracket on the homology of the free loop space $LM$ of an orientable $d$-manifold $M$ and showed that these operations make the shifted homology $H_{*+d}(LM)$ into a Gerstenhaber algebra. Moreover, this structure extends to a BV-algebra structure, where the BV-operator is induced by the $S^1$-action. A number of calculations have been made for specific classes of manifolds (see e.g. \cites{CohenJonesYan,Menichi,Tamanoi,Vaintrob,Hepworth,ChataurLeBorgne}). In this paper we consider manifolds that are highly connected relative to their dimension. Unless otherwise specified we take homology and cohomology with coefficients in a field $\kk$ of arbitrary characteristic.

\begin{theorem} \label{thm:main1}
Let $n\geq 2$ and suppose that $M$ is an $(n-1)$-connected closed manifold of dimension $d\leq 3n-2$ such that $\dim H^*(M) > 4$. Choose a basis $x_1,\ldots,x_r$ for the indecomposables of $H^*(M)$ and let $c_{ij} = \langle x_ix_j,[M]\rangle$.

The homology of the based loop space $U:=H_*(\Omega M)$, with the Pontryagin product, is freely generated as an associative algebra by classes $u_1,\ldots,u_r$, whose homology suspensions are dual to the classes $x_1,\ldots,x_r$ (in particular $|u_i| = |x_i|-1$), modulo the single quadratic relation
$$\sum_{i,j} (-1)^{|x_i|}c_{ji} u_iu_j = 0.$$

There is an isomorphism of Gerstenhaber algebras,
$$H_{*+d}(LM) \cong \kk \oplus s^{-1} \Der U / \ad U \oplus s^{-d} U/[U,U],$$
where the structure on the right hand side is specified as follows: Write $\theta,\eta,\ldots$ for elements of $\Der U/\ad U$ and $u,v,\ldots$ for elements of $U/[U,U]$. The product is given by
$$s^{-1}\theta \loopr s^{-1}\eta = s^{-d}\sum_{i,j} \pm c_{ij} \theta(u_i)\eta(u_j),\quad s^{-1}\theta \loopr s^{-d}u = s^{-d}u \loopr s^{-d}v = 0.$$
The Gerstenhaber bracket is given by
\begin{align*}
\{ s^{-1}\theta, s^{-1}\eta \} & = s^{-1} \theta \circ \eta - (-1)^{|\theta||\eta|} s^{-1}\eta \circ \theta,\\
\{ s^{-1}\theta, s^{-d}u\} & = s^{-d}\theta(u), \\
\{s^{-d}u,s^{-d}v\} & = 0.
\end{align*}

Moreover, the weight-graded vector spaces $s^{-1}\Der U / \ad U$ and $s^{-d}U/[U,U]$ are degreewise isomorphic in weight $3$ and above. In characteristic zero, the BV-operator is given by
$$\Delta(s^{-d}u_{i_1}\dots u_{i_r})=s^{-1}\sum_{k,\ell}\pm s^{-1}c_{i_k\ell}^{-1} u_{i_{k+1}}\dots u_{i_{k-1}}\frac{\partial}{\partial u_\ell},$$
for $u_{i_1}\dots u_{i_r}\in U/[U,U]$, and zero otherwise. In this situation, the BV-operator yields an isomorphism between $s^{-d}U/[U,U]$ and $s^{-1}\Der U / \ad U$ in weight $3$ and above.
\end{theorem}

Perhaps more interesting than the result itself are the techniques we are using. Our approach is algebraic and we use that $H_{*+d}(LM)$ may be calculated in terms of Hochschild cohomology (see Remark \ref{rem:disclaimer} below for a discussion about this). We also make heavy use of \emph{Koszul algebras} and the fact that highly connected manifolds are both \emph{formal} and \emph{coformal} (see \S\ref{sec:formal} for the definitions of these notions). Koszul algebras were first introduced by Priddy \cite{Priddy} as a tool for studying the cohomology of the Steenrod algebra. In this paper, we take advantage of the connection between formality, coformality and Koszul algebras that was elucidated in \cite{Berglund} to produce small chain algebra models for Hochschild cochains. Recall that the \emph{transgression} is the additive relation $\tau\colon H_*(X)\rightharpoonup H_{*-1}(\Omega X)$ induced by the differential $d^p\colon E_{p,0}^p\to E_{0,p-1}^p$ in the Serre spectral sequence of the path-loop fibration (see e.g.~\cite{MacLane}).

\begin{theorem} \label{thm:small-intro}
Let $\kk$ be a field and let $X$ be a simply connected space of finite $\kk$-type. If $X$ is both formal and coformal over $\kk$, then the transgression admits a lift to a twisting morphism $\tau\colon H_*(X) \to H_{*-1}(\Omega X)$ such that the twisted convolution algebra
$$\Hom^\tau(H_*(X),H_*(\Omega X))$$
is dga quasi-isomorphic to the Hochschild cochains of $C_*(\Omega X)$.
\end{theorem}
See Definitions \ref{def:twisting morphism} and \ref{def:twisted convolution} for the definitions of twisting morphisms and the twisted convolution algebra. Being simultaneously formal and coformal is a rather restrictive condition on a space, but there are many interesting examples appearing `in nature' apart from highly connected manifolds, see \cite{Berglund}. In fact, when the results of this paper were announced by the first author at the conference `Loop spaces in Geometry and Topology' in Nantes on September 4th 2014, we learned that Kallel and Salvatore \cite{KallelSalvatore} are using similar techniques to calculate the free loop space homology of (ordered) configuration spaces of points in $\RR^n$. In a forthcoming paper \cite{BerglundBorjeson}, we consider free loop space homology of certain moment-angle manifolds.

One of the main motivations for studying the homology of free loop spaces comes from the connection to closed geodesics (see e.g.~\cites{FelixOpreaTanre,GoreskyHingston}). The following result is a consequence of our explicit calculations, and verifies a conjecture of Gromov (see \cite{FelixOpreaTanre}*{Conjecture 5.3}) for the class of highly connected manifolds considered here.

\begin{theorem} \label{thm:geodesics}
Let $\kk$ be a field and let $M$ be an $(n-1)$-connected closed manifold of dimension at most $3n-2$ ($n\geq 2$) with $\dim H^*(M;\kk)>4$. For a generic metric on $M$, the number of geometrically distinct closed geodesics of length $\leq T$ grows exponentially in $T$.
\end{theorem}

Free loop space homology of simply connected closed 4-manifolds has been studied in \cite{BasuBasu}, but the methods used there do not extend to higher dimensions. Theorem \ref{thm:geodesics} generalizes \cite{BasuBasu}*{Theorem C(1)}. Free loop space homology of $(n-1)$-connected $2n$-dimensional manifolds has been studied in \cite{BebenSeeliger} using different methods, but the calculations there are not complete.


\begin{remark} \label{rem:disclaimer}
Our claims in Theorem \ref{thm:main1} rely on certain identifications of $H_{*+d}(LM)$ with Hochschild cohomology. We should spell out exactly what we are using. It is well-known that $H_{*+d}(LM)$ is isomorphic to the Hochschild cohomology of the singular cochain algebra $C^*(M)$, at least as a graded algebra \cite{CohenJones}. For simply connected $M$ and with coefficients in a field of arbitrary characteristic, F\'elix-Menichi-Thomas \cite{FelixMenichiThomas} have shown that there is an isomorphism of Gerstenhaber algebras between $HH^*(C^*(M),C^*(M))$ and $HH^*(C_*(\Omega M),C_*(\Omega M))$. In effect, it is this Gerstenhaber algebra we compute. According to the PhD thesis \cite{Malm}, there is an isomorphism of Gerstenhaber algebras $H_{*+d}(LM) \cong HH^*(C_*(\Omega M),C_*(\Omega M))$, without any restriction on the coefficients. In characteristic zero, F\'elix-Thomas \cite{FelixThomas}, building on \cite{Tradler}, extend the Gerstenhaber algebra structure on the Hochschild cohomology $HH^*(C^*(M),C^*(M))$ to a BV-algebra structure, and construct an isomorphism of BV-algebras $HH^*(C^*(M),C^*(M))\cong H_{*+d}(LM)$. It is this BV-operator we compute. In view of Menichi's calculation of $H_{*+2}(S^2;\FF_2)$ \cite{Menichi}, one should be careful about the BV-operator in positive characteristics.
\end{remark}

\subsection*{Conventions}
Unless otherwise specified, we work over a field $\kk$ of arbitrary characteristic. A chain complex is a $\ZZ$-graded $\kk$-vector space $A = \{A_n\}_{n\in \ZZ}$ with a differential $d_A\colon A_n\to A_{n-1}$ of degree $-1$. We use the convention $A^n = A_{-n}$ and think of cochain complexes as negative chain complexes. If $V$ is a graded vector space and $k$ is an integer, then we let $s^kV$ denote the graded vector space with $(s^k V)_i = V_{i-k}$. All unadorned tensor products are over $\kk$, i.e., $\tensor = \tensor_\kk$. As usual, the tensor product of two chain complexes $A$ and $B$ is defined by
$$(A\tensor B)_n = \bigoplus_{p+q=n} A_p\tensor B_q,$$
with differential $d_{A\tensor B} = d_A\tensor 1 + 1\tensor d_B$, and the chain complex $\Hom(A,B)$ is defined by
$$\Hom(A,B)_n = \prod_{p+q=n} \Hom(A^p,B_q),$$
with differential $\partial(f) = d_B \circ f - (-1)^{|f|} f\circ d_A$. By a \emph{dga} we mean a differential graded augmented associative algebra. We usually denote the structure maps of a dga $A$ by $\mu_A\colon A\tensor A\to A$ (multiplication), $\eta_A\colon \kk\to A$ (unit) and $\epsilon_A\colon A\to \kk$ (augmentation). Similarly, by a \emph{dgc} we mean a differential graded coassociative coaugmented coalgebra. We will mostly be concerned with dgas that are non-negatively graded and \emph{connected} in the sense that $H_0(A)\cong \kk$, or negatively graded and \emph{simply connected} in the sense that $H^0(A) \cong \kk$ and $H^1(A) = 0$. Similarly, we will mostly work with dgcs $C$ that are non-negatively graded and simply connected, in the sense that $H_0(C) \cong \kk$ and $H_1(C) = 0$.

\section{Formality, coformality and Koszul algebras}
In this section we will review the notions of formality, coformality and Koszul algebras and the ``2-out-of-3'' property for these notions \cite{Berglund}, along the way introducing notation and definitions that we will need in later sections.

\subsection{Twisting morphisms}
We begin by reviewing some facts about twisting morphisms (or twisting cochains). Standard references are \cite{HusemollerMooreStasheff}, \cite{LodayVallette} or \cite{Neisendorfer}.

\begin{definition}
Let $C$ be a dgc with comultiplication $\Delta_C\colon C\to C\tensor C$ and let $A$ be a dga with multiplication $\mu_A\colon A\tensor A\to A$. The \emph{convolution algebra} is the chain complex $\Hom(C,A)$ together with the \emph{convolution product} (or \emph{cup product}),
$$f\star g = \mu_A \circ (f\tensor g) \circ \Delta_C.$$
The unit is the map $\eta_A\circ \epsilon_C$ and the augmentation $\Hom(C,A)\to \kk$ is adjoint to the map $\eta_C \circ \epsilon_A$.
\end{definition}

\begin{definition} \label{def:twisting morphism}
An element $\tau$ in $\Hom(C,A)$ of degree $-1$ is called a \emph{twisting morphism} if it satisfies the Maurer-Cartan equation $$\partial(\tau)+\tau\star\tau = 0,$$
and if it is zero when composed with the (co)unit or (co)augmentation maps.
\end{definition}
The set of twisting morphisms $\Tw(C,A)$ is the value at $(C,A)$ of a bifunctor, contravariant in $C$ and covariant in $A$. Both functors $\Tw(-,A)$ and $\Tw(C,-)$ are representable; there are universal twisting morphisms $\pi\colon BA \to A$ and $\rho \colon C\to \Omega C$ that give rise to natural bijections
$$
\xymatrix{\Hom_{dga}(\Omega C,A) \ar[r]^-{\rho^*} & \Tw(C,A) & \ar[l]_-{\pi_*} \Hom_{dgc}(C,BA).}
$$
The representing objects $BA$ and $\Omega C$ are the classical bar and cobar constructions.

%
%
%

\begin{definition}
Given a twisting morphism $\tau$, the \emph{twisted tensor product} $C\otimes_\tau A$ is the tensor product of graded vector spaces with the differential $d:=d_{C\otimes A}+d_\tau$, where $d_{C\otimes A}$ is the usual differential on the tensor product of chain complexes and
$$d_\tau:=(Id_C\otimes \mu_A) \circ (Id_C\otimes \tau\otimes Id_A)\circ (\Delta_C \otimes Id_A).$$
\end{definition}


\begin{theorem} \label{thm:koszul twisting}
The following are equivalent for a twisting morphism $\tau\colon C\to A$.
\begin{enumerate}
\item The twisted tensor product $C\tensor_\tau A$ is contracible.
\item The dga morphism $\phi_\tau \colon \Omega C \to A$ is a quasi-isomorphism.
\item The dgc morphism $\psi_\tau \colon C\to BA$ is a quasi-isomorphism.
\end{enumerate}
\end{theorem}

\begin{definition} \label{def:koszul twisting morphism}
A twisting morphism $\tau\colon C\to A$ is called a \emph{Koszul twisting morphism} if the equivalent conditions in Theorem \ref{thm:koszul twisting} are fulfilled.
\end{definition}

\subsection{Koszul algebras}
Koszul algebras were introduced by Priddy \cite{Priddy}*{\S2}. For introductory accounts, see e.g.~\cites{PolishchukPositselski,LodayVallette}.

\begin{definition} \label{ref:koszul algebra}
A \emph{quadratic algebra} is a graded algebra $A$ that admits a presentation $A\cong TV/(R)$, where $V$ is a graded vector space of finite type and $(R)$ is the two-sided ideal in the tensor algebra $TV$ generated by a subspace $R\subseteq V^{\tensor 2}$. Since the relations are homogeneous, we may equip $A$ with an extra grading induced by the tensor length in $V$. This extra grading is inherited by the cohomology $\Ext_A^*(\kk,\kk) = H^*(\Hom(BA,\kk))$. By definition, a \emph{Koszul algebra} is a quadratic algebra $A$ such that $\Ext_A^{s,t}(\kk,\kk) = 0$ if $s\ne t$. There is a similar definition of Koszul coalgebras, see \cite{LodayVallette}.
\end{definition}

There is a variety of techniques for checking whether an algebra is Koszul without having to know $\Ext_A^*(\kk,\kk)$ beforehand, e.g., the PBW-criterion~\cite{Priddy}*{\S5}. The principal feature of Koszul algebras is that one can read off a presentation for the cohomology algebra $\Ext_A^*(\kk,\kk)$ by simple linear algebra.

\begin{definition}
Two quadratic algebras $A = TV/(R)$ and $B=TW/(S)$ are said to be \emph{Koszul dual} if there is a non-degenerate pairing of degree $+1$,
$$\langle -,-\rangle \colon W\tensor V \to \kk,$$
such that the subspaces $S\subseteq W^{\tensor 2}$ and $R\subseteq V^{\tensor 2}$ are orthogonal complements of one another under the induced pairing (of degree 2)
$$\langle - ,- \rangle\colon W^{\tensor 2} \tensor V^{\tensor 2} \to \kk,\quad \langle f\tensor g, u\tensor v\rangle = \pm \langle f,u\rangle \langle g, v\rangle.$$

Every quadratic algebra $A$ admits a unique up to isomorphism Koszul dual, denoted $A^!$. Clearly, $(A^!)^! \cong A$, because both are Koszul dual to $A^!$.
\end{definition}

\begin{theorem}[Priddy {\cite{Priddy}}]
If $A$ is Koszul then $\Ext_A^*(\kk,\kk)$ is isomorphic to $A^!$ as a graded algebra.
\end{theorem}

\begin{remark}
Given a quadratic algebra $A$, we let $A^{\as}$ denote the linear dual coalgebra of $A^!$. There is a twisting morphism $\kappa\colon A^{\as} \to A$, defined as the composite $A^{\as} \to W^* \cong V \to A$, where $W^*\cong V$ is the degree $-1$ isomorphism induced by the non-degenerate pairing. It is a basic fact that a quadratic algebra $A$ is a Koszul algebra, in the sense of Definition \ref{ref:koszul algebra}, if and only if the associated twisting morphism $\kappa\colon A^{\as} \to A$ is a Koszul twisting morphism in the sense of Definition \ref{def:koszul twisting morphism}. (As far as we understand, this is the reason for the name `\emph{Koszul} twisting morphism').
\end{remark}

Priddy's definition of a Koszul algebra may seem somewhat unsatisfactory, because it is not a priori clear whether the Koszul property depends on the choice of presentation for the algebra. The next theorem gives an intrinsic characterization of the Koszul property expressed without reference to any choice of presentation. Recall that a dga (dgc) is called \emph{formal} if it is quasi-isomorphic to its own homology, viewed as a dga (dgc) with trivial differential.

\begin{theorem}[Berglund {\cite{Berglund}}] \label{thm:formal-coformal-koszul}
Consider a dgc $C$ and a dga $A$ and suppose that $\tau\colon C\to A$ is a Koszul twisting morphism. The following are equivalent:
\begin{enumerate}
\item Both $A$ and $C$ are formal.
\item The dga $A$ is formal and $H_*(A)$ is a Koszul algebra.
\item The dgc $C$ is formal and $H_*(C)$ is a Koszul coalgebra.
\end{enumerate}
When the conditions hold, $H_*(C)$ is isomorphic to the Koszul dual coalgebra of $H_*(A)$.
\end{theorem}

\begin{remark}
The paper \cite{Berglund} is written with the assumption that the field $\kk$ has characteristic zero, but this restriction is unnecessary if one works with algebras over non-symmetric operads, such as associative algebras.
\end{remark}

\begin{remark}
Suppose that $A$ is a dga with trivial differential. When applied to the universal twisting morphism $\pi\colon BA \to A$, Theorem \ref{thm:formal-coformal-koszul} says that $A$ is a Koszul algebra if and only if the bar construction $BA$ is a formal dgc. This gives an intrinsic characterization of the Koszul property that is not expressed in terms of any presentation of $A$. (Note however that Priddy's notion of an \emph{inhomogeneous} Koszul algebra is not an intrinsic property of the algebra; it is a property of the chosen presentation.)
\end{remark}

\subsection{Formality and coformality for topological spaces} \label{sec:formal}
In this section, we will apply the algebraic results of the previous section to the dgc $C_*(X)$ and the dga $C_*(\Omega X)$ associated to a based topological space $X$. Here $C_*(-)$ stands for normalized singular chains with coefficients in $\kk$ and $\Omega X$ is the based loop space.
\begin{definition}
Let $\kk$ be a field and let $X$ be a based topological space.
\begin{enumerate}
\item We say that $X$ is \emph{formal over $\kk$} if the dgc $C_*(X)$ is formal.
\item We say that $X$ is \emph{coformal over $\kk$} if the dga $C_*(\Omega X)$ is formal.
\end{enumerate}
\end{definition}

Recall that we have the \emph{homology suspension},
$$\sigma_*\colon H_{*-1}(\Omega X) \to H_*(X),$$
which may be defined as the homomorphism $H_{*-1}(\Omega X) \cong H_*(\Sigma \Omega X) \to H_*(X)$ induced by the canonical map $\Sigma \Omega X\to X$. It is well-known that $\sigma_*$ vanishes on elements that are decomposable with respect to the Pontryagin product and that every class in the image of $\sigma_*$ is primitive (see \cite{Whitehead}*{Chapter VIII}). Therefore, $\sigma_*$ induces a well-defined pairing (of degree +1) on indecomposables,
\begin{equation*}
QH_*(\Omega X) \tensor QH^*(X)\to \kk,\quad \langle \alpha,x \rangle = \langle \sigma_*(\alpha), x \rangle,
\end{equation*}
which we will refer to as the \emph{homology suspension pairing}.

\begin{theorem} \label{thm:formal-coformal} \label{thm:fck-space}
Let $\kk$ be a field and let $X$ be a simply connected space of finite $\kk$-type. Consider the following statements:
\begin{enumerate}
\item The space $X$ is both formal and coformal over $\kk$.
\item The space $X$ is formal over $\kk$ and $H^*(X)$ is a Koszul algebra.
\item The space $X$ is coformal over $\kk$ and $H_*(\Omega X)$ is a Koszul algebra.
\item The homology suspension pairing is non-degenerate, both algebras $H^*(X)$ and $H_*(\Omega X)$ are Koszul algebras and they are Koszul dual via the homology suspension pairing;
$$H_*(\Omega X) \cong H^*(X)^!.$$
\end{enumerate}
The first three statements are equivalent and imply the fourth.
\end{theorem}
\begin{proof}
We may pass to a 1-reduced simplicial set model $K$ for $X$. For such $K$, Szczarba \cite{Szczarba} has constructed an explicit Koszul twisting morphism $C_*(K) \to C_*(GK)$. Here, $GK$ denotes the Kan loop group of $K$, which is a simplicial group model for the based loop space $\Omega X$. The result then follows by applying Theorem \ref{thm:formal-coformal-koszul}. The statement about the homology suspension follows from the fact that it may be realized as the map induced in homology by the projection from the cobar construction $\Omega C_*(K)$ to $s^{-1}C_*(K)$.
\end{proof}

\begin{remark}
It is conceivable that the fourth condition in Theorem \ref{thm:fck-space} implies the other three. We have not been able to find a counterexample.
\end{remark}

Being simultaneously formal and coformal is a rather restrictive constraint, but there are several interesting examples of spaces that fulfill it, see \cite{Berglund}. We will see in Section \ref{sec:fchcm} below that highly connected manifolds are formal and coformal over any field.
\begin{remark}
By a celebrated result due to Deligne-Griffiths-Morgan-Sullivan (see \cite{DGMS}), every compact K\"ahler manifold is formal over $\RR$ (and hence over any field of characteristic zero). If, in addition, the cohomology ring is a Koszul algebra then it is also coformal. It would be interesting to find a geometric characterization of what K\"ahler manifolds have this property.
\end{remark}
One might ask whether coformality together with Koszul cohomology implies formality, or whether formality together with Koszul loop space homology implies coformality. The following examples show that none of these implications hold.
\begin{example}
Consider the example of a non-formal closed simply connected $7$-manifold $M$ described in \cite{FelixOpreaTanre}*{Example 2.91}. Its minimal model is
$$\big(\Lambda(a,b,u,v,t),da=0,db=0,du=a^2,dv=b^2,dt=ab\big),$$
where $|a|=|b|=2$ and $|u|=|v|=|t|=3$. Formality is obstructed by the non-zero Massey operations $x=\langle a,b,b\rangle$ and $y = \langle a,a,b\rangle$, but the rational cohomology algebra is Koszul, because it admits the quadratic presentation
$$H^*(M;\QQ) \cong \Lambda(a,b,x,y)/(a^2,ab,b^2,ay,bx,ax-by),\quad |a|=|b|=2,\,\, |x|=|y|=5,$$
and it is easy to see that $1,a,b,x,y,ax$ is a PBW-basis. On the other hand, $M$ is coformal, because the minimal model has purely quadratic differential. In fact, the minimal model is isomorphic to the Chevalley-Eilenberg cochain algebra of the graded Lie algebra $\LL(\alpha,\beta)/([\alpha,[\alpha,\beta]],[\beta,[\alpha,\beta]])$, where $|\alpha|=|\beta|=1$, so it follows that the homology of the based loop space is the cubic algebra
$$H_*(\Omega M;\QQ) \cong \QQ\langle \alpha,\beta\rangle/([\alpha,[\alpha,\beta]],[\beta,[\alpha,\beta]]),\quad |\alpha|=|\beta|=1.$$
Thus, $M$ is an example of a coformal but non-formal manifold that has Koszul cohomology, but non-Koszul loop space homology.
\end{example}

\begin{example}
Complex projective space $\CP{n}$ is formal over $\QQ$, e.g., because it is a K\"ahler manifold. The cohomology algebra
$$H^*(\CP{n};\QQ) \cong \QQ[x]/(x^{n+1}),\quad |x|=2,$$
is not quadratic, and hence not Koszul, if $n\geq 2$. This implies that $\CP{n}$ is not coformal. However, the homology of the based loop space is a free graded commutative algebra,
$$H_*(\Omega \CP{n};\QQ) \cong \Lambda(\alpha)\tensor \QQ[\beta],\quad |\alpha|=1,\,\,|\beta|=2n,$$
which is Koszul. Thus, $\CP{n}$ ($n\geq 2$) is an example of a formal but non-coformal manifold, with Koszul loop space homology but non-Koszul cohomology.
\end{example}

\section{Hochschild cohomology}
In this section we explain how to construct a small dga model for the Hochschild cochains of the chain algebra $C_*(\Omega X)$ when $X$ is a formal and coformal space, by exploiting the connection to Koszul algebras discussed in the previous section. Small cochain complexes for computing the Hochschild cohomology of a Koszul algebra, such as the one described in Theorem \ref{hochschildcocomplextheorem} below, are presumably well-known. What is new here is the application to formal and coformal spaces and the interpretation of the twisting morphism as a lift of the transgression. We also discuss how to calculate the Gerstenhaber algebra structure and, when $X$ is a Poincar\'e duality space, the BV-algebra structure.

\begin{definition} \label{def:twisted convolution}
Let $\tau\colon C\to A$ be a twisting morphism. We define the \emph{twisted convolution algebra} to be the dga
$$\Hom^\tau(C,A) = \big(\Hom(C,A),\star,\partial^\tau \big),$$
with differential $\partial^\tau = \partial + [\tau,-]$, where, as usual,
$$\partial(f) = d_A \circ f - (-1)^f f\circ d_C,\quad [\tau,f] = \tau \star f -(-1)^{|f|}f\star \tau.$$
\end{definition}
The Maurer-Cartan equation for $\tau$ ensures that $\partial^\tau$ squares to zero, and it is easy to verify that $\partial^\tau$ is a derivation with respect to the convolution product.

\begin{observation}
For every dga $A$, the twisted convolution algebra associated to the universal twisting morphism $\pi\colon BA\to A$ is isomorphic, as a dga, to the standard Hochschild cochain complex $C^*(A,A)$ computing $HH^*(A,A)$, with the cup product;
$$\Hom^\pi(BA,A) \cong C^*(A,A).$$
\end{observation}

\begin{theorem} \label{hochschildcocomplextheorem}
Let $A$ be a Koszul algebra with Koszul twisting morphism $\kappa\colon A^{\as}\rightarrow A$. The associated quasi-isomorphism $\psi_\kappa\colon A^{\as} \to BA$ induces a quasi-isomorphism
$$\psi_\kappa^*\colon \Hom^\pi(BA,A) \xrightarrow{\sim} \Hom^\kappa(A^{\as},A)$$
of dgas. In particular, there is an isomorphism of graded algebras
$$H^*(\Hom^\kappa(A^{\as},A)) \cong HH^*(A,A).$$
\end{theorem}

\begin{proof}
Write $f=\psi_\kappa$. Since $A$ is Koszul, there is a contraction
$$\bigSDR{\big(BA,d_{BA}\big)}{\big(A^{\as},0\big)}{g}{f}{h},$$
where $f$ is a dgc morphism. Applying the functor $\Hom(-,A)$, we obtain a contraction
$$\bigbigSDR{\big(\Hom(BA,A),\partial\big)}{\big(\Hom(A^{\as},A),0\big)}{f^*}{g^*}{h^*},$$
where $f^*$ is a dga morphism. Consider now the perturbation $t=[\pi,-]$ of the chain complex $\big(\Hom(BA,A),\partial\big)$, where $\pi$ is the universal twisting morphism. Applying the basic perturbation lemma with $t$ as initiator, we obtain a new contraction
$$\bigbigSDR{\big(\Hom(BA,A),\partial + t\big)}{\big(\Hom(A^{\as},A),t'\big)}{f'}{g'}{h'},$$
see Theorem \ref{basicperturbationlemma}. To see that the sum $\sum_{n\geq 0}(h^*t)^n$ converges, we use the fact that the algebra $A$ carries a weight-grading. The chain complex $\Hom(BA,A)$ inherits a filtration from this grading, and it is easy to see that $t$ increases the filtration degree while $h^*$ preserves it. It follows that $\sum_{n\geq 0}(h^*t)^n$ converges point-wise. Since $f^*$ is a dga morphism, the formulas for $f'$ and $t'$ simplify. Indeed, $f'$ is given explicitly by $f' = f^*+f^*th^*+f^*th^*th^*+\dots$, where $f^*th^*=f^*[\pi,h^*]=[f^*\pi,f^*h^*]=0$ since $f^*$ is an algebra morphism and $f^*h^*=0$ since it is a contraction. Thus $f'=f^*$ and in particular it is also an algebra morphism. Next, $t'=f^*tg^*+f^*th^*tg^*+\dots$, where $f^*tg^*=[f^*\pi,f^*g^*]=[\kappa,-]$. The higher terms all vanish in the same way as above, so we may identify $\big(\Hom(A^{\as},A),t'\big)$ with $\Hom^\kappa(A^{\as},A)$. Thus, we see that $f'=f^*$ is a dga quasi-isomorphism from $\Hom^\pi(BA,A) = \big(\Hom(BA,A),\partial + t\big)$ to $\Hom^\kappa(A^{\as},A)$.
\end{proof}

\begin{theorem} \label{thm:small}
Let $\kk$ be a field and let $X$ be a simply connected space of finite $\kk$-type. If $X$ is formal and coformal over $\kk$, then the transgression lifts to a Koszul twisting morphism $\tau\colon H_*(X)\to H_*(\Omega X)$ such that the twisted convolution algebra $\Hom^\tau(H_*(X),H_*(\Omega X))$ is dga quasi-isomorphic to the Hochschild cochain complex of $C_*(\Omega X)$.
\end{theorem}

\begin{proof}
Since $X$ is coformal, the Hochschild cochain algebra of $C_*(\Omega X)$ is quasi-isomorphic to that of $H_*(\Omega X)$, and then the claim follows from Theorem \ref{thm:formal-coformal} and Theorem \ref{hochschildcocomplextheorem}. Concerning the statement about the transgression: Under the stated hypotheses, $H_*(X)$ is a Koszul coalgebra and $H_*(\Omega X)$ is its Koszul dual algebra. In particular, there are decompositions
$$H_*(X) = \kk \oplus H_{*,1}(X) \oplus H_{*,2}(X) \oplus \dots,$$
$$H_*(\Omega X) = \kk \oplus H_{*,1}(\Omega X) \oplus H_{*,2}(\Omega X) \oplus \dots,$$
compatible with the coalgebra and algebra structures. We may identify the primitives $PH_*(X)$ with $H_{*,1}(X)$ and the indecomposables $QH_*(\Omega X)$ with $H_{*,1}(\Omega X)$. The homology suspension induces an isomorphism $\sigma\colon QH_*(\Omega X)\cong PH_*(X)$. The inverse is given by the transgression $\tau\colon  PH_*(X) \to QH_*(\Omega X)$ (under the stated hypotheses the transgression is actually a well-defined homomorphism with domain $PH_*(X)$ and codomain $QH_*(\Omega X)$). We may extend $\tau$ to a map $H_*(X)\to H_*(\Omega X)$ simply by letting it be zero on $H_{*,k}(X)$ for $k\ne 1$.
\end{proof}

\begin{remark} \label{rem:finite}
If $X$ is formal and coformal and of finite type over $\kk$, then we may choose a basis $x_1,\dots,x_r$ for the indecomposables of $H^*(X)$, and a basis $u_1,\dots,u_r$ for the indecomposables of $H_*(\Omega X)$, such that the homology suspension of $u_i$ is dual to $x_i$. If, in addition, the cohomology $H^*(X)$ is finite dimensional, then there is an isomorphism of dgas
$$\Hom^\tau(H_*(X),H_*(\Omega X))  \cong \big( H^*(X)\tensor H_*(\Omega X) , [\kappa,-]\big),$$
where the underlying algebra of the dga on the right hand side is simply the tensor product of the algebras $H^*(X)$ and $H_*(\Omega X)$, and the differential $[\kappa,-]$ is given by taking the commutator with the element
$$\kappa = x_1\tensor u_1 + \dots + x_r \tensor u_r \in H^*(X)\tensor H_*(\Omega X).$$
\end{remark}

\subsection{The Gerstenhaber algebra structure}
Gerstenhaber \cite{Gerstenhaber} observed that the Hochschild cohomology $HH^*(A,A)$ of an associative algebra $A$ carries a Lie bracket of degree $1$ that interacts well with the cup product. The structure is now called a \emph{Gerstenhaber algebra}.

\begin{definition}
A Gerstenhaber algebra is a graded commutative algebra together with skew-symmetric binary bracket $[,]$ raising degree by $1,$ satisfying the Jacobi identity and being a derivation of the product in both variables. 
\end{definition}

We will now show how to calculate the Gerstenhaber bracket on Hochschild cohomology in terms of certain dg Lie algebras of derivations. This is not a new idea, it is essentially dual to \cite{Stasheff}, but we have not found precisely the statements we need in the literature.

\begin{definition}
Let $f\colon A\to A'$ be a morphism of dgas. The chain complex of $f$-derivations $\Der_f(A,A')$ is defined to be the subcomplex of $\Hom(A,A')$ whose elements are the maps $\theta\colon A\to A'$ that satisfy
$$\theta \circ \mu_A = \mu_{A'}\circ (\theta \tensor f + f\tensor \theta),$$
where $\mu_A\colon A\tensor A \to A$ and $\mu_{A'}\colon A'\tensor A' \to A'$ are the multiplication maps. If $f$ is the identity map on $A$, then we write $\Der A$ for $\Der_f(A,A)$. The graded commutator
$$[\theta,\eta] = \theta \circ \eta - (-1)^{|\theta||\eta|} \eta \circ \theta$$
makes $\Der A$ into a dg Lie algebra.

Similarly, if $g\colon C\to C'$ is a dgc morphism, then the chain complex of $g$-coderivations $\Coder_g(C,C')$ is defined to be the subcomplex of $\Hom(C,C')$ whose elements are the maps $\theta\colon C\to C'$ that satisfy
$$\Delta_{C'} \circ \theta = (\theta \tensor g + g\tensor \theta)\circ \Delta_C.$$
If $g$ is the identity map on $C$, then we write $\Coder C$ for $\Coder_g(C,C)$. As before, the graded commutator makes $\Coder C$ into a dg Lie algebra.
\end{definition}

Let $\tDer_f(A,A')$ denote the chain complex $\Der_f(A,A')\oplus sA$ with differential
$$D(sa') = \ad_{a'} - sd_{A'}(a'),$$
where $\ad_{a'}$ is the $f$-derivation of degree $|a'|$ given by
$$\ad_{a'}(a) = [a',f(a)].$$
For $a\in A$, let $\omega_a\colon C\to A$ denote the map of degree $|a|$ given by $\omega_a(\lambda) = \lambda a$ for $\lambda\in \kk$ and $\omega_a(x) = 0$ for $x\in \overline{C}$.

\begin{lemma} \label{lemma:id}
Let $\tau\colon C\to A$ be a twisting morphism and let $\psi_\tau\colon \Omega C \to A$ be the unique dga morphism such that $\psi_\tau \circ \rho = \tau$, where $\rho\colon C\to \Omega C$ denotes the universal twisting morphism. The map
$$
\rho^*\colon \tDer_{\psi_\tau}(\Omega C,A) \to s\Hom^\tau(C,A),
$$
$$\theta \mapsto (-1)^{|\theta|}s(\theta \circ \rho),\quad sa \mapsto s\omega_a,$$
is an isomorphism of chain complexes.
\end{lemma}

\begin{proof}
The underlying algebra of $\Omega C$ is the tensor algebra on $s^{-1}\overline{C}$, so it is clear that the map is a bijection. We leave the straightforward verification that $\rho^*$ commutes with the differentials to the reader.
\end{proof}

\begin{proposition} \label{prop:lie bracket}
Let $A$ be a Koszul algebra and let $\kappa\colon A^{\as} \to A$ be the associated Koszul twisting morphism. The isomorphism
$$H_*(\tDer \Omega A^{\as}) \cong HH^*(A,A),$$
induced by the quasi-isomorphism of chain complexes $\tDer \Omega A^{\as} \to s\Hom^\kappa(A^{\as},A)$, is an isomorphism of graded Lie algebras.
\end{proposition}

\begin{proof}
Stasheff \cite{Stasheff} observed that there is an isomorphism of chain complexes
$$\tCoder BA \cong s\Hom^{\pi}(BA,A)$$
such that the Lie bracket on $\tCoder BA$ corresponds to the Gerstenhaber bracket in cohomology. By Theorem \ref{hochschildcocomplextheorem}, and its dual version, there are surjective quasi-isomorphims of chain complexes
$${\tiny \tCoder BA\cong s\Hom^{\pi}(BA,A) \xrightarrow{\psi_\kappa^*} s\Hom^{\kappa}(A^{\as},A) \xleftarrow{(\phi_\kappa)_*} s\Hom^{\rho}(A^{\as},\Omega A^{\as}) \cong \tDer \Omega(A^{\as})}$$
We need to show that the two Lie brackets on the cohomology of $\Hom^{\kappa}(A^{\as},A)$ induced from the Lie brackets on $\tCoder BA$ and $\tDer \Omega(A^{\as})$, respectively, coincide. To see that this is the case, form the pullback
$$
\xymatrix{\tDer \Omega A^{\as} \ar[r] & s\Hom^\kappa(A^{\as},A) \\
\mathcal{L} \ar[r] \ar[u] & \tCoder BA .\ar[u]}
$$
Here $\mathcal{L}$ is the dg Lie algebra whose elements are pairs $(\theta,\eta)$, where $\theta\in \tDer \Omega A^{\as}$ and $\eta\in \tCoder BA$ are (co)derivations such that $(\phi_\tau)_*\rho^*(\theta) = \psi_\tau^* \pi_*(\eta)$. Differentials and Lie brackets are computed componentwise; in particular the maps from $\mathcal{L}$ to $\tDer \Omega A^{\as}$ and $\tCoder BA$ are morphisms of dg Lie algebras. Since the diagram is a pullback, and since the morphisms with target $s\Hom^\kappa(A^{\as},A)$ are surjective quasi-isomorphisms, it follows that the maps from $\mathcal{L}$ to $\tDer \Omega A^{\as}$ and $\tCoder BA$ are surjective quasi-isomorphisms of dg Lie algebras. This implies that the two Lie brackets in the cohomology of $s\Hom^\kappa(A^{\as},A)$ induced from $\tDer \Omega A^{\as}$ and $\tCoder BA$ are the same.
\end{proof}

\subsection{The Batalin-Vilkovisky algebra structure}
When $A$ is a Frobenius \\ (Poincar\'e duality) algebra, then the Gerstenhaber algebra structure on Hochschild cohomology $HH^*(A,A)$ can be enhanced to a BV-algebra structure.

\begin{definition}
A Batalin-Vilkovisky algebra (BV-algebra) is a Gerstenhaber algebra together with a square-zero unary operator $\Delta$ of degree $+1$ such that $$[a,b]=\Delta(ab)-\Delta(a)b-(-1)^{|a|}a\Delta(b).$$
\end{definition}

In \cite{Tradler}, a Batalin-Vilkovisky structure is put on the Hochschild cohomology of an algebra equipped with a non-degenerate bilinear form. The result is proved in greater generality, but we state a simpler version to avoid more definitions.

\begin{theorem}[See {\cite{Tradler}}] \label{tradlertheorem}
Let $\kk$ be a field of characteristic $0$ and let $A$ be a finite dimensional, graded, unital associative algebra equipped with a graded symmetric invariant non-degenerate bilinear form $A\otimes A\rightarrow \kk.$ Then there is a Batalin-Vilkovisky structure defined as follows. Suppose $f\in \Hom^\pi(BA,A)$ with support on the weight $n$ part. Then $\Delta f$ is the unique function with support on the weight $n-1$ part such that
$$\langle \Delta f(a_1,\dots,a_{n-1}),a_n \rangle = \langle \sum_{i=1}^n \pm f(a_i,\dots,a_n,a_1,\dots,a_{i-1}),1\rangle,$$
where $\pm$ is a Koszul sign coming from permutation of the elements, remembering that they have been suspended, explicitly given as $(-1)$ to the power
$$(|a_1|-1)(|a_2|-1+\dots +|a_n|-1)+\dots +(|a_i|-1)(|a_{i+1}|-1+\dots +|a_{i-1}|-1).$$
\end{theorem}

\section{Highly connected manifolds}
In this section we will apply the results of the previous sections to highly connected manifolds.

\subsection{Formality and coformality of highly connected manifolds} \label{sec:fchcm}
It is well known that every $(n-1)$-connected space $X$ of dimension at most $3n-2$ is formal over $\QQ$ (see e.g.~\cite{FelixOpreaTanre}*{Proposition 2.99}). Neisendorfer and Miller \cite{NeisendorferMiller} observed that a closed manifold with the same connectivity and dimension constraints is also coformal over $\QQ$, provided the cohomology has rank $>3$. In this section we generalize these result to fields of arbitrary characteristic, using Koszul algebras and Theorem \ref{thm:formal-coformal-koszul}.

\begin{theorem} \label{thm:formal}
Let $\kk$ be a PID and $n\geq 2$. If $X$ is an $(n-1)$-connected space such that $H^i(X;\kk) = 0$ for all $i>3n-2$ and $H_i(X;\kk)$ is a free $\kk$-module for all $i$, then $X$ is formal over $\kk$.
\end{theorem}

\begin{proof}
Since both $C_*(X;\kk)$ and $H_*(X;\kk)$ are degreewise free as $\kk$-modules, and $\kk$ is a PID, it follows that the chain complex $C_*(X;\kk)$ is split. Hence, the cochain complex $C^*(X;\kk)$ is split as well. In other words, it is possible to find a contraction
$$
\bigbigSDR{C^*(X;\kk)}{H^*(X;\kk)}{f}{g}{h},
$$
where $dh+hd=1-gf$, $fg = 1$, and $fh=0,$ $hh=0$, $hg=0$. We may apply the homotopy transfer theorem to obtain an $A_\infty$-structure $\{m_i\}_{i\geq 2}$ on $H^*(X;\kk)$, such that $m_2$ is the standard cup product in cohomology, and $(H^*(X;\kk),\{m_i\})$ is $A_\infty$-equivalent to $C^*(X;\kk)$. By studying the explicit formulas for the transferred structure (see Section \ref{appendix:ainfty}), we see that $m_i(\dots,1,\dots)=0$ for all $i\geq 3$, because each term in the formula will contain $fh$, $hh$ or $hg$, which is zero. Next, let $i\geq 3$ and suppose that $x_1,\dots,x_i\in H^*(X;\kk)$ are non-zero classes of positive degree. Since we assume that $X$ is $(n-1)$-connected, we must have $|x_j|\geq n$ for all $j$. Hence,
$$|m_i(x_1,\dots,x_i)| = 2-i + |x_1|+\cdots +|x_i| \geq 2-i + in\geq 3n-1.$$
But we are assuming that $H^{\geq 3n-1}(X;\kk) = 0$, so $m_i(x_1,\ldots,x_i)$ is necessarily zero.
\end{proof}

\begin{theorem} \label{thm:koszul}
Let $\kk$ be a field and let $n\geq 2$. Suppose that $M$ is an $(n-1)$-connected closed manifold of dimension $d\leq 3n-2$. Then the cohomology algebra $H^*(M;\kk)$ is a Koszul algebra if and only if $\dim_\kk H^*(M;\kk)\ne 3$.
\end{theorem}

\begin{proof}
Let $r+2 = \dim_\kk H^*(M;\kk)$. If $r=0$, then $H^*(M;\kk)\cong \kk\langle x\rangle/(x^2)$, $|x|=d$, which is a Koszul algebra. If $r=1$, then $d$ is necessarily even and $H^*(M;\kk) \cong \kk\langle x \rangle/(x^3)$, $|x|=d/2$. This algebra is not Koszul because it does not admit any quadratic presentation.

Next, let $r\geq 2$. Suppose that we can find a non-zero class $x\in H^k(M;\kk)$, for some $k<d$, such that $x^2=0$. Then by Poincar\'e duality, we can find a class $x'\in H^{d-k}(M;\kk)$ such that $x'x$ is a generator for $H^d(M;\kk)$ (see e.g.~\cite{Hatcher}*{Corollary 3.39}). Setting $x_{r-1} = x'$ and $x_r = x$, we can complete to a basis for $H^*(M;\kk)$ of the form
$$1,x_1,\ldots,x_{r-1},x_r,x_rx_{r-1}.$$
If we declare $x_rx_{r-1}$ to be the only admissible monomial, then the displayed basis is a PBW-basis in the sense of Priddy \cite{Priddy}*{\S5}. This implies that $H^*(M;\kk)$ is Koszul. Let us point out that unless $x_r^2= 0$, the above is not a PBW-basis.

So we would like to find a non-zero cohomology class $x\in H^k(M;\kk)$, for some $k<d$, such that $x^2=0$. Unless $d=2k$, it is automatic that $x^2 = 0$ because of the connectivity and dimension constraints. So we are done unless $d$ is even, say $d=2k$, and the only non-zero cohomology is in degrees $0$, $k$, $2k$. If $k$ is odd and $\kk$ is not of characteristic $2$, then $x^2 = 0$ is automatic because of the graded commutativity of the cup product. Otherwise, the cup product defines a symmetric bilinear form on $H^k(M;\kk)$, and we can find an orthogonal basis $e_1,\ldots,e_r$, such that $e_ie_j = \delta_{ij}y$, where $y$ is some chosen generator for $H^d(M;\kk)$ (see e.g.~\cite{HusemollerMilnor}). Finding $x = \lambda_1e_1+\cdots+\lambda_re_r\ne 0$ such that $x^2=0$ is then equivalent to finding a non-trivial solution to the equation
$$\lambda_1^2+\cdots +\lambda_r^2 = 0$$
in $\kk$. If $r\geq 2$ and $\kk$ has a square root of $-1$, then $\lambda_1=1,\lambda_2=\sqrt{-1},\lambda_3=\cdots=\lambda_r=0$ is a solution. Otherwise, one might get stuck. Fortunately, the property of being Koszul is preserved and reflected under field extensions, so we may extend scalars to $\kk[\sqrt{-1}]$ and use the same argument to conclude that $H^*(M;\kk[\sqrt{-1}]) \cong H^*(M;\kk)\tensor_\kk \kk[\sqrt{-1}]$ is Koszul, and hence that $H^*(M;\kk)$ is Koszul as well.

%
%
%
%
%
\end{proof}

\begin{remark}
The equation $\lambda_1^2 + \lambda_2^2 =0$ has no non-trivial solutions over $\FF_3$, so the above argument will fail to produce a PBW-basis for the algebra $\FF_3[x,y]/(x^2-y^2,xy)$. In fact, it is possible to show that this algebra does not admit any PBW-basis at all. But it is Koszul.
\end{remark}

\begin{remark}
To see that the hypothesis on the dimension of $H^*(M;\kk)$ is necessary we can look at the $1$-connected $4$-manifold $\mathbb{C}\mathbb{P}^2$. This space is formal, but its cohomology algebra does not admit any quadratic presentation, so it cannot be Koszul.
\end{remark}

%
%
%
%

\begin{corollary} \label{cor:main}
Let $\kk$ be a field and let $n\geq 2$. Suppose that $M$ is an $(n-1)$-connected closed manifold of dimension at most $3n-2$ such that $\dim_\kk H^*(M;\kk) \ne 3$. Then $M$ is both formal and coformal. If we choose a basis $x_1,\ldots,x_r$ for the indecomposables of $H^*(M;\kk)$ and let $c_{ij}$ represent the intersection form in this basis, i.e., $\langle x_ix_j, [M] \rangle  = c_{ij}$, then the loop space homology algebra of $M$ admits the presentation
$$H_*(\Omega M;\kk) \cong \kk\langle u_1,\ldots,u_r\rangle/(\omega),\quad \omega = \sum_{i,j} (-1)^{|x_i|}c_{ji}u_i u_j,$$
where the homology suspension of $u_i$ is dual to $x_i$. In particular, $|u_i| = |x_i| - 1$.
\end{corollary}

\begin{proof}
By Theorem \ref{thm:formal}, the manifold $M$ is formal over $\kk$. By Theorem \ref{thm:koszul}, the cohomology ring $H^*(M;\kk)$ is a Koszul algebra. It follows from Theorem \ref{thm:formal-coformal} that $M$ is coformal, and that $H_*(\Omega M;\kk)$ may be calculated as the Koszul dual of $H^*(M;\kk)$. The cohomology $H^*(M;\kk)$ admits a quadratic presentation of the form
$$\kk\langle x_1,\ldots,x_r\rangle/(R),$$
where $R$ is spanned by all graded commutators $x_ix_j-(-1)^{|x_i||x_j|}x_jx_i$ and all elements of the form $c_{ij}x_kx_\ell - c_{k\ell}x_ix_j$. We know that the homology suspension is non-degenerate; choose a basis $u_1,\dots, u_r$ for the indecomposables $W = QH_*(\Omega X;\kk)$ dual to $x_1,\dots,x_r$ under the homology suspension pairing. Since $H^*(M;\kk) \cong TV/(R)$ and we know that $V^{\tensor 2}/R \cong H^d(M;\kk)$ is one-dimensional, the orthogonal subspace $R^\perp\subseteq W^{\tensor 2}$ must be one-dimensional. One checks that the element
$$\omega = \sum_{i,j} (-1)^{|x_i|}c_{ji}u_iu_j\in W^{\tensor 2}$$
is orthogonal to $R$, so it must generate $R^\perp$.
\end{proof}

\subsection{Homology of the free loop space} \label{explicitcomplex}
In this section, we will use Theorem \ref{thm:small} and the results of the previous section to calculate the free loop space homology of highly connected manifolds.

Before we can state the result, we need to recall some facts about graded derivations. For a graded algebra $U$, we let $\Der U$ denote the graded vector space of derivations of $U$. Its elements of degree $k$ are the linear maps $\theta \colon U \to U$ of degree $k$ such that $\theta(\alpha\beta) = \theta(\alpha)\beta + (-1)^{|\alpha|k} \alpha \theta(\beta)$ for all $\alpha,\beta \in U$. The graded commutator $[\theta,\eta] = \theta \circ \eta -(-1)^{|\theta||\eta|}\eta \circ \theta$ makes $\Der U$ into a graded Lie algebra.

The algebra $U$ may itself be viewed as a graded Lie algebra with the commutator Lie bracket $[\alpha,\beta] = \alpha\beta - (-1)^{|\alpha||\beta|} \beta \alpha$. The equality
$$[\xi,\alpha\beta] = [\xi,\alpha]\beta + (-1)^{|\xi||\alpha|} \alpha[\xi,\beta],$$
for $\xi,\alpha,\beta \in U$, shows that every element $\xi$ in $U$ defines a derivation $\ad_\xi = [\xi,-]$ of degree $|\xi|$. Derivations of this form are called \emph{inner derivations}. The equality
$$[\theta,\ad_\xi] = \ad_{\theta(\xi)},$$
for $\theta \in \Der U$ and $\xi\in U$, shows that the subspace $\ad U \subseteq \Der U$ spanned by all inner derivations is a Lie ideal. The quotient $\Der U / \ad U$ is the graded Lie algebra of \emph{outer derivations} on $U$.

\begin{theorem} \label{thm:algebra}
Let $\kk$ be a field and let $n\geq 2$. Let $M$ be an $(n-1)$-connected closed manifold of dimension $d\leq 3n-2$, such that $\dim_\kk H^*(M) > 4$, and let $U = H_*(\Omega M)$. There is an isomorphism of Gerstenhaber algebras
$$HH^*(C_*(\Omega M),C_*(\Omega M)) \cong \kk \oplus s^{-1} \Der U / \ad U \oplus s^{-d} U/[U,U].$$
The unit element of the left summand $\kk$ acts as a unit for the multiplication. Given two outer derivations $\theta$ and $\eta$, the product of their images in $s^{-1}\Der U / \ad U$ is given by
$$s^{-1}\theta \bullet s^{-1}\eta = s^{-d}\sum_{i,j} (-1)^{\epsilon} c_{ij}\theta(u_i)\eta(u_j) \in s^{-d} U/[U,U],$$
where the sign in the above sum is given by
$$\epsilon = |\theta|(|u_i|+|x_j|) + |\eta||u_j| + |x_j||u_i|.$$
The product of $s^{-d}u$ with anything except multiples of the unit element is zero. The Gerstenhaber bracket is given by
$$\{s^{-1}\theta,s^{-1}\eta\} = s^{-1} \{\theta,\eta\},\quad \{s^{-1}\theta,s^{-d}u\} = s^{-d} \theta(u), \quad \{s^{-d}u,s^{-d}v\} = 0.$$
\end{theorem}

\begin{proof}
By Corollary \ref{cor:main} and Theorem \ref{thm:small} (see also Remark \ref{rem:finite}), the Hochschild cochain complex of $C_*(\Omega M)$ is quasi-isomorphic, as a dga, to
$$\big(H^*(M)\tensor H_*(\Omega M), [\kappa,-]\big),\quad \kappa = x_1\tensor u_1 +\cdots + x_r\tensor u_r,$$
where $x_1,\ldots,x_r$ is a basis for the indecomposables of $H^*(M;\kk)$ and $u_1,\dots,u_r$ is the dual basis for the indecomposables of $H_*(\Omega M)$. Write $A=H^*(M)$ and $U = H_*(\Omega M)$. We may decompose $A$ as
$$A = A(0) \oplus A(1) \oplus A(2),$$
where $A(0)\cong \kk$ is spanned by the unit element, $A(1)$ is spanned by $x_1,\ldots,x_r$ and $A(2) = H^d(M)$ is one-dimensional. The differential $[\kappa,-]$ then acts as follows
$$A(0)\tensor U \xrightarrow{[\kappa,-]} A(1)\tensor U \xrightarrow{[\kappa,-]} A(2)\tensor U.$$
Let $d_i = |x_i|$. By inspection, the above chain complex is isomorphic to
$$U \xrightarrow{\partial_1} s^{-d_1} U \oplus \cdots \oplus s^{-d_r} U \xrightarrow{\partial_0} s^{-d} U,$$
$$\partial_1(\xi) = \big( [u_1,\xi],\ldots,[u_r,\xi] \big), \quad \partial_0(\zeta_1,\ldots,\zeta_r) = \sum_{i,j} (-1)^{|x_i|}c_{ij} [u_j,\zeta_i].$$
Clearly, the kernel of $\partial_1$ is the center, $Z(U)$, of $U$. If $\dim_\kk H^*(M) > 4$, then the center is trivial; $Z(U) = \kk$, by \cite{Bogvad}.


Since the matrix $(c_{ij})$ is invertible, the image of $\partial_0$ is spanned by all commutators in $U$ of the form $[u_i,\xi]$. By using the relation 
$$[\alpha\beta,\gamma] = [\alpha,\beta \gamma] + (-1)^{|\alpha||\beta| + |\alpha||\gamma|} [\beta,\gamma\alpha]$$
and the fact that $u_1,\ldots,u_r$ generate $U$ as an algebra, one sees that the image of $\partial_0$ is in fact equal to the subspace $[U,U]$ spanned by all commutators in $U$.

The middle homology may be identified with the space of outer derivations on $U$. Indeed, by evaluating derivations $\theta\colon U\to U$ on the algebra generators, we get a map
$$ev \colon \Der U \to \ker \partial_0,\quad ev(\theta) = \big((-1)^{|\theta||\alpha_1|}\theta(\alpha_1),\ldots,(-1)^{|\theta||\alpha_r|}\theta(\alpha_r) \big).$$
A calculation shows that
$$\partial_0(ev(\theta)) = (-1)^{d(|\theta|-1)} \theta(\omega) = 0,$$
so that the image of $ev$ is really in $\ker \partial_0$.

Under the identification $\ker \partial_0 \cong \Der U$, the image of $\partial_1$ may be identified with the subspace $\ad U\subseteq \Der U$ consisting of inner derivations, i.e., derivations of the form $\theta = [-,\xi]$.

The algebra structure is induced from the tensor product of the algebras $A\tensor U$, and it is straightforward to derive the description of the product stated in the theorem. The proof that the description of the Gerstenhaber bracket is correct is a little more subtle and will be given in the next section.
\end{proof}

\subsection{The Gerstenhaber bracket}
We use the fact that the Gerstenhaber bracket in Hochschild cohomology may be computed via the Lie bracket of derivations in $\tDer \Omega C$ (Proposition \ref{prop:lie bracket}).
Here $C= H_*(M;\kk)$ is the homology coalgebra. Denote the cobar construction by $\UU = \Omega C$. It admits the following explicit description:
$$\UU = \big(T(\alpha_1,\ldots,\alpha_r,\gamma), \delta(\gamma) = -\omega \big), \quad |\alpha_i| = |x_i|-1, \,\, |\gamma| = d-1.$$
Let $f\colon \UU\to U$ denote the quasi-isomorphism that sends $\alpha_i$ to $u_i$ and $\gamma$ to zero.

The chain complex $\tDer \UU$ is spanned by three types of elements:
$$
s\xi,\quad \xi \frac{\partial}{\partial\alpha_i},\quad \xi \frac{\partial}{\partial \gamma},
$$
where $\xi \in \UU$. The differential $D$ is described by the following:
\begin{align} \label{eq:der-diff}
D(s\xi) & = \sum_{i=1}^r [\xi,\alpha_i] \frac{\partial}{\partial \alpha_i} + [\xi,\gamma]\frac{\partial}{\partial \gamma} - s\delta(\xi), \\
D\big(\xi \frac{\partial}{\partial \alpha_i}\big) & = \delta(\xi)\frac{\partial}{\partial \alpha_i} - (-1)^{|\xi|} \sum_{j=1}^r c_{ji}[\xi,\alpha_j]\frac{\partial}{\partial \gamma}, \\
D\big(\xi \frac{\partial}{\partial \gamma}\big) & = \delta(\xi) \frac{\partial}{\partial \gamma}.
\end{align}

We know that the homology of $\tDer_f(\UU,U) \cong s\big(A\tensor U,[\kappa,-]\big)$ may be represented by two types of classes: outer derivations, represented by elements of the form
$$\sum_{i=1}^r \xi_i \frac{\partial}{\partial \alpha_i},\quad \xi_i \in U$$
such that \eqref{eq:cycle} holds, and elements of $U/[U,U]$, represented by
$$\zeta \frac{\partial}{\partial \gamma},\quad \zeta \in U.$$
To calculate their Lie brackets, we need to find cycle representatives in $\tDer \UU$, compute their Lie bracket in $\tDer \UU$ and then apply $f_*$.

For the second type of elements, we may take any pre-image $\overline{\zeta}\in \kk\langle \alpha_1,\ldots,\alpha_r\rangle$ of $\zeta$; the derivation
$$\overline{\zeta} \frac{\partial}{\partial \gamma}\in \tDer \UU$$
is then a cycle that maps to $\zeta \frac{\partial}{\partial \gamma}$ under $f_*$. Finding cycle pre-images of the first type of elements is a little trickier. We use the following lemma.

\begin{lemma} \label{lemma:rep}
Consider the surjective quasi-isomorphism
$$\tDer \UU \xrightarrow{f_*} \tDer_f(\UU,U).$$
For every positive degree cycle $\theta \in \tDer_f(\UU,U)$ of the form
$$\theta = \sum_{i=1}^r \xi_i\frac{\partial}{\partial \alpha_i},\quad \xi_i\in U,$$
it is possible to choose a cycle pre-image in $\tDer \UU$ of the form
$$\overline{\theta} = \sum_{i=1}^r \overline{\xi}_i\frac{\partial}{\partial \alpha_i} + \eta \frac{\partial}{\partial \gamma},$$
where $\overline{\xi}_i\in \kk\langle \alpha_1,\ldots,\alpha_r\rangle$ are pre-images of $\xi_i$ and $\eta$ belongs to $[\UU,\UU]$.
\end{lemma}

\begin{proof}
That $\theta$ is a cycle means that the equality
\begin{equation} \label{eq:cycle}
\sum_{i,j} (-1)^{|\xi_j|}c_{ji} [\xi_i,u_j] = 0
\end{equation}
holds in $U$. Choose elements $\overline{\xi}_i \in \kk\langle \alpha_1,\ldots,\alpha_r\rangle$ such that $f(\overline{\xi}_i) = \xi_i$. Then the element
$$\zeta = \sum_{i,j} (-1)^{|\xi_i|}c_{ji} [\overline{\xi}_i,\alpha_j]$$
is a cycle in $[\UU,\UU]$, and it belongs to the kernel of $f$ because of the equality \eqref{eq:cycle}. Since $f\colon \UU\to U$ is a quasi-isomorphism, it follows that there must be an element $\eta\in\UU$ of $\gamma$-degree $1$ such that $\delta(\eta) = \zeta$. In fact, we may choose $\eta\in [\UU,\UU]$, because the homology of the chain complex $\UU/[\UU,\UU]$ in $\gamma$-degree $1$ is spanned by the class of $\gamma$. In view of the formula \eqref{eq:der-diff} for the differential in $\tDer \UU$, it follows that
$$\overline{\theta} = \sum_{i=1}^r \overline{\xi}_i\frac{\partial}{\partial \alpha_i} + \eta \frac{\partial}{\partial \gamma}$$
is a cycle in $\tDer \UU$. Moreover, it maps to $\theta$ under $f_*$ because $f(\overline{\xi}_i) = \xi_i$ and $f(\eta) = 0$.
\end{proof}

Finally, we can calculate; if we take two elements of the first form $\theta$ and $\theta'$ and choose cycle pre-images $\overline{\theta}$ and $\overline{\theta}'$ as in Lemma \ref{lemma:rep}, a calculation shows that $f_*[\overline{\theta},\overline{\theta}']$ equals the class of the commutator $[\theta,\theta']$ in $\Der U$.

If we take the bracket of the two cycle representatives
$$\sum_{i=1}^r \overline{\xi}_i\frac{\partial}{\partial \alpha_i} + \eta \frac{\partial}{\partial \gamma}\quad \mbox{and}\quad \overline{\zeta} \frac{\partial}{\partial \gamma},$$
then we obtain the expression
$$\overline{\zeta} \frac{\partial}{\partial \gamma}(\eta) \frac{\partial}{\partial \gamma} + \sum_{i=1}^r \overline{\xi}_i\frac{\partial}{\partial \alpha_i}(\overline{\zeta}) \frac{\partial}{\partial \gamma}.$$
The second term corresponds to the action of $\theta\in\Der U/\ad U$ on $\zeta\in U/[U,U]$. The first term vanishes in $U/[U,U]$, because $\eta\in [\UU,\UU]$ (this is why we needed to pay extra attention to $\eta$ in Lemma \ref{lemma:rep}).

\subsection{The BV-operator}
We continue to determine the BV-operator, at least when $\kk$ has characteristic zero. To state the result, we need to introduce some more notation. The small complex computing $H_{*+d}(LM;\kk)$ can be further decomposed as follows:
$$A(0)\tensor U(r) \xrightarrow{[\kappa,-]} A(1)\tensor U(r+1) \xrightarrow{[\kappa,-]} A(2)\tensor U(r+2).$$
We will denote the homology group corresponding to $A(i)\otimes U(r)$ by $H_{i,r}.$ We will also use the notation from Theorem \ref{thm:small} instead of talking about derivations in order to make the calculations more transparent. Given an element $x_i\otimes y\in A(1)\otimes U$ we can view it as the derivation $s^{-1}\frac{\partial}{\partial u_i}y=\pm s^{-1}y\frac{\partial}{\partial u_i}.$ An element of $s^{-d}U/[U,U]$ will be represented as $M^\vee\otimes y$ for some $y\in U.$ Here $M^\vee$ is a choice of generator for $H^d(M;\kk)$.

The goal of this section is to prove the following theorem. 

\begin{theorem}
Let $\kk$ be a field of characteristic zero and let $M$ be an $(n-1)$-connected manifold of at most dimension $3n-2.$
\begin{enumerate}
\item There is a map $$\Delta:H_{2,r}\rightarrow H_{1,r-1}$$ defined by sending an element $M^\vee\otimes u_{i_1}\dots u_{i_{r}}$ of $A(2)\otimes U(r)$ to $$\sum_{k,\ell}\pm c_{i_k\ell}^{-1}x_\ell \otimes u_{i_{k+1}}\dots u_{i_{k-1}}$$ where $\pm$ is given by the Koszul sign rule. 
\item By setting $\Delta=0$ on all other basis elements this gives us the Batalin-Vilkovisky structure.
\item When $r\geq 3,$ $\Delta$ gives us an isomorphism $$H_{2,r}\cong H_{1,r-1}.$$
\end{enumerate}
\end{theorem}

\begin{proof}
This is proved later in the section. To avoid too painful obfuscated Koszul sign computations we will do the proof for $(n-1)$-connected $2n$-manifolds where $\pm= (-1)^{(n-1)(r-1)(k-1)}c_{i_k\ell}^{-1}$. Items (1) and (2) are Lemma \ref{BVlemma} and item (3) is Proposition \ref{isomorphismproposition}.
\end{proof}





In arbitrary characteristic there is also an isomorphism $H_{2,r}\cong H_{1,r-1}$ when $r\geq 3.$ We will begin this section by proving this. Our main tool are Hilbert series so we need some results about these.

%

\begin{definition}
Suppose $A$ is a quadratic algebra, then there is a weight grading on $A$ corresponding to the number of generators in an element. Denote the weight $n$ part by $A_{(n)}.$ Define the Hilbert series of $A$ by $$f^{A}(t):=\sum_{i=0}^\infty \dim(A_{(n)})t^n.$$ Define the Hilbert series $f_C(t)$ of a quadratic coalgebra in the same way.
\end{definition}

\begin{proposition}
Suppose we have a Koszul morphism $\kappa:C\rightarrow A$ between a graded coalgebra and a graded algebra (in particular, $C$ and $A$ have zero differential). Then the equation $$f_C(t)f^A(-t)=1$$ holds.
\end{proposition}

\begin{proof}
By assumption the complex $C\otimes_\kappa A$ is contractible. It splits in subcomplexes according to weight grading. The Euler characteristic of the weight $n$ part is calculated by $$\sum_{i=0}^n(-1)^{n-i}\dim(C_{(i)})\dim(A_{(n-i)}).$$ This sum is zero unless the weight is $0$, where the sum is $1.$ Putting this together yields the result.
\end{proof}

If we dualize the coalgebra we obtain the following result.

\begin{proposition}
\label{koszulhilbertlemma}
For a Hilbert series $f^A(t)$ associated to a Koszul algebra and a Hilbert series $f^{A^{\text{!`}}}(t)$ associated to its Koszul dual algebra we have the relation
$$f^A(t)f^{A^{\text{!`}}}(t)=1.$$
\end{proposition}

Now we can start applying this to our situation.

\begin{lemma}
\label{hilbertserieslemma}
The Hilbert series of the algebra $A$ is $1+nt+t^2$ and the series of $U$ is $\frac{1}{1-nt+t^2}.$ Ler $R$ be the module of relations in $U.$ The series of the module $R$ is $$\frac{t^2}{1-(2nt-(n^2+1)t^2+nt^3)}.$$
\end{lemma}

\begin{proof}
The Hilbert series of $A$ is directly from the definition. The series of $U$ then follows from Lemma \ref{koszulhilbertlemma} since $A$ and $U$ are Koszul dual to each other. The Hilbert series of the tensor algebra is $\frac{1}{1-nt}.$ Now since $R(r)=T(r)/U(r)$ the Hilbert series for $R$ is $$\frac{1}{1-nt}-\frac{1}{1-nt+t^2}=\frac{t^2}{(1-nt)(1-nt+t^2)}=\frac{t^2}{1-(2nt-(n^2+1)t^2+nt^3)}.$$
\end{proof}

\begin{lemma}
\label{dimensionlemma}
Let $V=A(1).$ Consider the vector spaces $V\otimes U(r+1)$ and $U(r+2)$ with $r\geq 0.$ We have the relation $\dim(V\otimes U(r+1))-\dim(U(r+2))=\dim(U(r)).$
\end{lemma}

\begin{proof}
The weight graded vector space $U$ has Hilbert series $u(t):=\frac{1}{1-nt+t^2}$ and $V\otimes U$ has $ntu(t)=\frac{nt}{1-nt+t^2}.$ To calculate the difference of dimensions we calculate the difference of Hilbert series. $$ntu(t)-u(t)=\frac{nt}{1-nt+t^2}-\frac{1}{1-nt+t^2}=\frac{-1+nt}{1-nt+t^2}=$$ $$\frac{-1+nt-t^2}{1-nt+t^2}+\frac{t^2}{1-nt+t^2}=-1+t^2u(t).$$ This shows that the difference of dimensions of $V\otimes U(r+1)$ and $U(r+2)$ are given by the dimension of $U(r)$ since that is the part of weight 2 less. 
\end{proof}

\begin{lemma}
\label{hilbertseriesdifferencelemma}
Suppose $r\geq 3$ and  $n\geq 3.$ There is an isomorphism $$H_{2,r}\cong H_{1,r-1}.$$ We also have $$\dim(H_{1,1})=\dim(H_{2,2})+1,$$ $$\dim(H_{0,0})=\dim(H_{2,0})=1$$ and $$\dim(H_{1,0})=\dim(H_{2,1})=n.$$
\end{lemma}

\begin{proof}
We will prove this by counting dimensions. By elementary linear algebra we have that the dimension of the kernel minus the dimension of the cokernel of the map $$A(1)\tensor U(r+1) \xrightarrow{[\kappa,-]} A(2)\tensor U(r+2)$$ is $\dim(A(1)\otimes U(r+1))-\dim(A(2)\otimes U(r+2)).$ By Lemma \ref{dimensionlemma} this difference is equal to the dimension of $U(r).$ 
The left map of the complex $$A(0)\tensor U(r) \xrightarrow{[\kappa,-]} A(1)\tensor U(r+1) \xrightarrow{[\kappa,-]} A(2)\tensor U(r+2)$$ is injective unless $r=0$ by \cite{Bogvad}. This means that the dimension of the image is $\dim(U(r))$ which shows that $H_{1,r+1}\cong H_{2,r+2}.$ The second formula is proved in the same way with the difference we use that the map is zero in the case $r=0.$ The other formulas follow easily from the fact that the corresponding differentials are zero.
\end{proof}

Now we are going to restrict to characteristic zero and provide a description of these homology groups and isomorphisms using the BV-operator. Note that the expression defining $\Delta$ makes even when we are not over a field of characteristic zero. One can ask if this gives the right BV-structure in the case of arbitrary characteristic as well.

\begin{lemma}
\label{Deltalemma}
The operator $\Delta$ satisfies the equation $$\Delta(M^\vee\otimes u_{i_1}\dots u_{i_{r-1}})(x_{i_r}\otimes u_{i_r})=\sum_{k=1}^r(-1)^{(r-1)(n-1)k}M^\vee\otimes u_{i_k}\dots u_{i_r} u_{i_1} \dots u_{i_{k-1}}.$$
\end{lemma}

\begin{proof}
We have $$\Delta(M^\vee\otimes u_{i_1}\dots u_{i_{r-1}})(x_{i_r}\otimes u_{i_r})=$$ $$(\sum_{k,\ell}(-1)^{(n-1)(r-1)(k-1)}c_{i_k\ell}^{-1}x_\ell \otimes u_{i_{k+1}}\dots u_{i_{k-1}})(x_{i_r}\otimes u_{i_r})=$$ $$\sum_{k,\ell}(-1)^{(n-1)(r-1)(k-1)}c_{i_k\ell}^{-1}c_{\ell i_{r}}M^\vee \otimes u_{i_{k+1}}\dots u_{i_{k-1}}u_{i_r}=$$ $$\sum_k (-1)^{(n-1)(r-1)(k-1)}M^\vee\otimes u_{i_{k+1}}\dots u_{i_{k-1}}u_{i_k}=$$ $$\sum_k (-1)^{(n-1)(r-1)k}M^\vee\otimes u_{i_k}u_{i_{k+1}}\dots u_{i_{k-1}}.$$
\end{proof}

\begin{lemma}
\label{BVlemma}
The operator $\Delta$ is the BV-operator.
\end{lemma}

\begin{proof}
In characteristic zero there is a BV-operator on Hochschild cohomology is defined by the equation $$<\Delta f(a_1,\dots,a_{n-1}),a_r>= <\sum_{k=1}^r (-1)^{(r-1)(n-1)k} f(a_i,\dots,a_r,a_1,\dots,a_{i-1}),1>,$$ by Theorem \ref{tradlertheorem} \cite{Tradler}. By \cite{FelixThomas} this coincides with the Chas-Sullivan BV-algebra structure on $H_*(LX).$ Using that the non-degenerate product is given by the multiplication and that we can pick any cycle representative this equation turns into $$\Delta(M^\vee\otimes u_{i_1}\dots u_{i_{r-1}})(x_{i_r}\otimes u_{i_r})=\sum_{k=1}^r(-1)^{(r-1)(n-1)k}M^\vee\otimes u_{i_k}\dots u_{i_r} u_{i_1} \dots u_{i_{k-1}},$$ which is satisfied by Lemma \ref{Deltalemma}.
\end{proof}

Next we would like to prove that $\Delta$ is an isomorphism in certain degrees; the strategy is to first observe that $\Delta$ is similar to first go from cyclic coinvariants to cyclic invariants by averaging and then apply an automorphism depending on the intersection form. By using the snake lemma we will compare the effect of doing this before and after quotioning $T(V)$ with the relations in $U.$

\begin{lemma}
\label{Clemma}
There is a vector space automorphism $C$ of $A(1)\otimes U(r+1)$ given by $$C(\sum_ix_i\otimes y_i)=\sum_{i,j}c_{ji}x_j\otimes y_i.$$
\end{lemma}

\begin{proof}
$C$ is an automorphism since $c_{ji}$ are the elements of an invertible matrix (it is invertible since it is the intersection form of an orientable manifold).
\end{proof}

\begin{lemma}
\label{Dequaldprime}
Let $D:A(1)\otimes U(r+1)\rightarrow A(2)\otimes U(r+2)$ be given by $D(x_k\otimes y)= M^\vee \otimes [u_k,y].$ We have 
$[\kappa,-]=D\circ C$ and thus $ker([\kappa,-])=C^{-1}(ker(D))$ and $im([\kappa,-])=im(D).$
\end{lemma}

\begin{proof} We have
$$D\circ C(x_\ell\otimes y)=D(\sum_k c_{kl}x_k\otimes y)=\sum_k c_{k\ell}M^\vee\otimes [u_k,y]= [\kappa,x_\ell\otimes y].$$
\end{proof}



\begin{lemma}
\label{Adlemma}
Let $V:=\kk\{x_1,\dots,x_m\}$ and $W:=\kk\{w_1,\dots,w_m\}.$ There is a map $Ad:V\otimes W^{\otimes r-1}\rightarrow W^{\otimes r}$ given by $Ad(v_i\otimes w_{j_1}\dots w_{j_{r-1}})=[w_i,w_{j_1}\dots w_{j_{r-1}}]=w_iw_{j_1}\dots w_{j_{r-1}}-(-1)^{(n-1)(r-1)}w_{j_1}\dots w_{j_{r-1}}w_i.$

This map induces maps to give a commutative diagram with exact rows as follows.

\begin{tikzpicture}[>=triangle 60]
\matrix[matrix of math nodes,column sep={85pt,between origins},row
sep={60pt,between origins},nodes={anchor=center}] (s)
{
|[name=00]|0 &|[name=A]| V\otimes R(r-1) &|[name=B]| V\otimes W^{\otimes r-1} &|[name=C]| V\otimes U(r-1) &|[name=01]| 0 \\
|[name=02]| 0 &|[name=A']| R(r) &|[name=B']| W^{\otimes r} &|[name=C']| U(r) & |[name=03] |0 \\
};
\draw[-cm to]  
          (A) edge (B)
          (B) edge (C)
          (C) edge (01)
          (A) edge node[auto] {Ad'} (A')
          (B) edge node[auto] {Ad} (B')
          (C) edge node[auto] {D} (C')
          (02) edge (A')
          (A') edge (B')
          (B') edge (C')
          (00) edge (A)
          (C') edge (03)
;
\end{tikzpicture}

We have denoted the rightmost map by $D$ since it coincides with the map $D$ of Lemma \ref{Dequaldprime} up to the isomorphism $A(2)\otimes U(r) \cong U(r).$

\end{lemma}

Here we are after computing the kernel and cokernel of $D,$ to do this we will first analyze the kernel and cokernel of $Ad.$

\begin{lemma}
\label{cyclicinvariantslemma}
The cyclic group of $r$ elements with generator $\sigma$ acts on $V\otimes W^{\otimes r-1}$ by $\sigma(v_i\otimes w_{j_1}\dots w_{j_{r-1}})=(-1)^{(n-1)(r-1)}v_{j_1}\otimes w_{j_2} \dots w_{j_{r-1}}w_i.$ Similarly it acts on $W^{\otimes r}$ by $\sigma(w_{j_1}w_{j_2}\dots w_{j_r})=(-1)^{(n-1)(r-1)}w_{j_2}\dots w_{j_r}w_{j_1}.$ Then $ker(Ad)$ are the invariants of the first action and $coker(Ad)$ are the coinvariants of the second action. We also have $ker(Ad)\cong coker(Ad).$
\end{lemma}

\begin{proof}
By inspection of the defining formulas of $Ad$ we see that $ker(Ad)=ker(1-\sigma)$ and $coker(Ad)=coker(1-\sigma).$ To see the isomorphism we first note that there is an isomorphism between $ker(Ad)$ and the invariants of the second action by taking $v_i\mapsto w_i.$ Then there is an isomorphism between invariants and coinvariants of the second action since they are the kernel and cokernel of the map $1-\sigma,$ respectively, and kernel and cokernels of endomorphisms are isomorphic as vector spaces.
\end{proof}

\begin{lemma}
\label{kercokerlemma}
There is a map $P:ker(Ad)\rightarrow coker(Ad)$ by first mapping $V$ to $W$ given by $x_i \mapsto w_i$ and then considering the equivalence class.
Over a field of characteristic zero there is an explicit inverse $Q/r$ where $$Q(w_{j_1}\dots w_{j_{r}})=\sum_{i=1}^r (-1)^{(i-1)(n-1)(r-1)}v_{j_i}\otimes w_{j_{i+1}}\dots w_{j_{i-1}}.$$ 
\end{lemma}

\begin{proof}
If one implicitly uses the isomorphism $V\cong W$ then this is just the usual statement that one can go from coinvariants to invariants by averaging over the group.
\end{proof}

The following combinatorial description will also be useful when we count the dimension of the homology groups.

\begin{lemma}
\label{kernelnecklacelemma}
Suppose that $n$ or $r$ is odd. Then $ker(Ad)$ has a basis with one basis element for each necklace (= word up to cyclic permutation) of length $r$ with letters in the set ${1,\dots, m}.$ If instead $n$ and $r$ are even there is a basis with one basis element for each necklace with even period length.
\end{lemma}

\begin{proof}
Suppose that $n$ or $r$ is odd. We would like to find the invariants of the action $\sigma(v_i\otimes w_{j_1}\dots w_{j_{r-1}})=v_{j_1}\otimes w_{j_2} \dots w_{j_{r-1}}w_i.$ Suppose $v_i\otimes w_{j_1}\dots w_{j_{r-1}}$ is a term of an element that is invariant. That element also have to contain the term $v_{j_1}\otimes w_{j_2} \dots w_{j_{r-1}}w_i.$ But then it also have to contain the term $v_{j_2}\otimes w_{j_3}\dots w_{j_1}.$ Continuing like this we have to include all cyclic permutations until we come back to the term we started with. This shows the first part. To prove the case when $n$ and $r$ are even we look at the action $\sigma(v_i\otimes w_{j_1}\dots w_{j_{r-1}})=-v_{j_1}\otimes w_{j_2} \dots w_{j_{r-1}}w_i.$ Suppose $v_i\otimes w_{j_1}\dots w_{j_{r-1}}$ is a term of an element that is invariant. Then that element has to include the term $-v_{j_1}\otimes w_{j_2} \dots w_{j_{r-1}}w_i.$ as well. In the same way it has to include $v_{j_2}\otimes w_{j_3}\dots w_{j_1}.$
Continuing like this we obtain an element in the kernel that is non-zero if the period length is even. If the period length is odd we see that we are forced to add terms such that we obtain a zero vector.
\end{proof}

\begin{lemma}
\label{snakelemmalemma}
We have the following diagram given by the snake lemma.

\begin{tikzpicture}[>=triangle 60]
\matrix[matrix of math nodes,column sep={72pt,between origins},row
sep={60pt,between origins},nodes={anchor=center}] (s)
{
|[name=000]|0 &|[name=ka]| A &|[name=kb]| B &|[name=kc]| \ker D \\
|[name=00]|0 &|[name=A]| V\otimes R(r-1) &|[name=B]| V\otimes W^{\otimes r-1} &|[name=C]| V\otimes U(r-1) &|[name=01]| 0 \\
|[name=02]| 0 &|[name=A']| R(r) &|[name=B']| W^{\otimes r} &|[name=C']| U(r) & |[name=03] |0 \\
&|[name=ca]| X &|[name=cb]| Y &|[name=cc]| \coker D & |[name=04]| 0\\
};
\draw[-cm to] (ka) edge (A)
          (kb) edge (B)
          (kc) edge (C)
          (A) edge (B)
          (B) edge (C)
          (C) edge (01)
          (A) edge node[auto] {Ad'} (A')
          (B) edge node[auto] {Ad} (B')
          (C) edge node[auto] {D} (C')
          (02) edge (A')
          (A') edge (B')
          (B') edge (C')
          (A') edge (ca)
          (B') edge (cb)
          (C') edge (cc)
          (000) edge (ka)
          (00) edge (A)
          (C') edge (03)
          (cc) edge (04)
;
\draw[-cm to] (ka) edge (kb)
               (kb) edge node[auto] {p} (kc)
               (ca) edge node[auto] {i} (cb)
               (cb) edge (cc)
;
\draw[-cm to,rounded corners] (kc) -| node[auto,text=black,pos=.7]
{\(\partial\)} ($(01.east)+(.5,0)$) |- ($(B)!.35!(B')$) -|
($(02.west)+(-.5,0)$) |- (ca);
\end{tikzpicture}

When $r\geq 3,$ $\partial=0,$ $i$ injective and $p$ surjective. If $r=2,$ we have $i=0.$

\end{lemma}

\begin{proof}
When $r=2,$ one sees that $Y\cong \coker D,$ showing that $i=0$ by exactness. Now suppose $r\geq3.$
We first want to prove that $i$ is injective. Note that $Y$ is the cyclic coinvariants of Lemma \ref{cyclicinvariantslemma}. The map $Ad':V\otimes R(r-1)\rightarrow R(r)$ is defined by $Ad'(v_i\otimes w_{j_1}\dots w_{j_{r-1}})=[w_i,w_{j_1}\dots w_{j_{r-1}}]=w_iw_{j_1}\dots w_{j_{r-1}}-(-1)^{(n-1)(r-1)}w_{j_1}\dots w_{j_{r-1}}w_i.$ The image gives us exactly the relations that enable us to permute cyclically (recall the sign in the definition of cyclic action) in the submodule $R(r)\subset W^{\otimes r}.$ From this description it is clear that two distinct elements in $X$ cannot be identified in $Y$ since there are no more relations than cyclic permutations. This shows that $i$ is injective and that $\partial=0$ and $p$ surjective follows from exactness.
\end{proof}

\begin{proposition}
\label{isomorphismproposition}
\begin{enumerate}
\item The map $\Delta$: $H_{2,r}\rightarrow H_{1,r-1}$ is well-defined.
\item When $r\geq 3,$ this map is an isomorphism and when $r=2$ it is injective.
\end{enumerate}
\end{proposition}

\begin{proof} That $\Delta$ is well-defined is proved in \cite{Tradler} but we prefer to reprove it in this special case since the computation helps us prove that it is an isomorphism in certain degrees.
The map $\Delta$ can be described as follows. Pick a representative of an element in $H_{2,r}.$ Consider it as an element in $\coker(D)$ in the diagram of Lemma \ref{snakelemmalemma}. Pick an element in $Y$ mapping to it. Apply $Q$ to get an element in $B.$ Map it to $\ker(D)$ with $p.$ Take the quotient with the image of $C\circ d_{r-2}$ and finally apply $C^{-1}.$ We will first treat the case $r\geq 3.$ To see that the map described is well defined we need to show that it is independent of the representative chosen in $Y.$ Equivalently we need to prove that any element $i(x)\in Y$ with $x\in X$ is mapped to $0.$ There is a basis of $X$ given by elements of the type $x=\omega\otimes w_{j_1}\dots w_{j_{r-2}}$ where $\omega\in W^{\otimes 2}$ is the relation of $U.$ By abuse of notation we will write $\omega$ for the corresponding element in $V\otimes W$ as well and also use Sweedler type notation $\omega=\omega_{(1)}\otimes \omega_{(2)}.$ Then $Qi(x)=$ $$\omega\otimes u_{j_1}\dots u_{j_{r-2}}+(-1)^{(n-1)(r-1)}\omega_{(2)}\otimes u_{j_1}\dots u_{j_{r-2}} \otimes \omega_{(1)}+\sum\pm \dots v_{r-2}\otimes \omega \otimes v_1 \dots$$ 
All terms except the first two vanish when we apply $p$ to land in $\ker(D).$ We want to see that these two terms vanish when we take the quotient with $C\circ d_{r-2}.$ Thus we look at the image of $u_{j_1}\dots u_{j_{r-2}} \in U(r-2)$ in $\ker(D).$ We see that $C\circ d_{r-2}=$ $$\sum_{i,j}c_{ji}x_i\otimes[u_j,u_{j_1}\dots u_{j_{r-2}}]=\pm2pQi(x).$$ Since we assumed that the characteristic was zero, $pQi(x)$ lies in the image and thus the map $\Delta$ is well defined. To see that the map is an isomorphism for $r\geq 3,$ it is enough to prove that it is injective since by Lemma \ref{hilbertseriesdifferencelemma} we know that the homology groups have the same dimension. Going backwards in the previous computation we see that the terms from $B$ that are mapped to zero in $H_{1,r-1}$ are exactly those containing $\omega$ in some way. These are precisely the ones mapping to $X\subset Y$ under $Q^{-1}.$ Thus none of the non-trivial elements coming from $\coker(D)$ maps to zero. This shows that the map is injective and therefore an isomorphism. To prove that $\Delta$ well-defined and injective when $r=2$ we observe that in this case $i=0,$ $A=0$ and $C\circ d_0=0.$
\end{proof}

\begin{proposition}
\label{r2proposition}
Suppose $char(k)=0.$
Let $V:=\kk\{x_1,\dots,x_m\}$ and $W:=\kk\{u_1,\dots,u_m\}.$ We have,
\begin{enumerate}
\item $H_{2,2}=W\otimes W/[W,W].$ 
\item $H_{1,1}= im(\Delta)+ \kk\{\kappa\}.$
\end{enumerate}
\end{proposition}

\begin{proof}
That $H_{2,2}$ is  $W/[W,W]$ follows from the same computation as in the proof of Theorem \ref{thm:algebra} if one also notes that $\omega$ is in $[W,W]$ (since $char(k)=0$). To compute the kernel of $$A(1)\tensor U(1) \xrightarrow{[\kappa,-]} A(2)\tensor U(2)$$  we have to work a bit more. First note that by Lemma \ref{dimensionlemma} we know that it has to have dimension one more than $H_{2,2}$. By Lemma \ref{isomorphismproposition} we know that $im(\Delta)$ sits in the kernel. We see that we only need one more generator to generate the whole kernel. Now we see that $[\kappa,\kappa]=2\omega=0,$ and note that $\kappa$ is not contained in $im(\Delta).$
\end{proof}



We would now like to describe the product structure in more detail using the description of $H_{2,r}$ as $U(r)/[U(r),U(r)]$ and the isomorphism of Proposition \ref{isomorphismproposition}.

\begin{lemma}
Suppose we have elements $a$ and $b.$ Their product is zero unless both of them come from groups $H_{1,*}$ or at least one of them come from $H_{0,0}.$
\end{lemma}

\begin{proof}
The multiplication respects the bidegree. This yields the result together with the observation that $H_{0,*}$ is one-dimensional, spanned by the identity element.
\end{proof}

First we will start by looking at the multiplication of two classes in $im(\Delta).$ This is the generic case, we only have to compute a couple of extra special cases to determine the whole product.

\begin{theorem}
\label{producttheorem}
Suppose we have elements in $H_{1,r_1}$ and $H_{1,r_2}$ given by $\Delta(M^\vee\otimes u_{i_1}\dots u_{i_{r_1+1}})$ and $\Delta(M^\vee\otimes u_{j_1}\dots u_{j_{r_2+1}})$  Then their product is given by $$\Delta(M^\vee\otimes u_{i_1}\dots u_{i_{r_1+1}})\Delta(M^\vee\otimes u_{j_1}\dots u_{j_{r_2+1}})=$$ $$(-1)^n\sum_{k,\ell}(-1)^{(n-1)r(k+\ell)}c^{-1}_{i_kj_\ell}M^\vee \otimes u_{i_{k+1}}\dots u_{i_{k-1}} u_{j_{\ell+1}}\dots u_{j_{\ell-1}},$$ where $c^{-1}_{i_kj_\ell}$ is a coefficient of the inverse matrix of $C.$
\end{theorem}

\begin{proof}
We compute by using the formula for $\Delta.$
$$(\sum_{k,s}(-1)^{\epsilon_1}c^{-1}_{i_ks}x_s\otimes u_{i_{k+1}}\dots u_{i_{k-1}})(\sum_{\ell,p}(-1)^{\epsilon_2}c^{-1}_{j_\ell p}x_p\otimes u_{j_{\ell+1}}\dots u_{j_{\ell-1}})=$$
$$\sum_{k,s,\ell,p}(-1)^{\epsilon_3}c^{-1}_{i_ks}c^{-1}_{j_\ell p}c_{s p}M^\vee \otimes u_{i_{k+1}}\dots u_{i_{r_1+1}}u_{i_1}\dots u_{i_{k-1}}u_{j_{\ell+1}}\dots u_{j_{r_2+1}}u_{j_1}\dots u_{j_{\ell-1}}=$$
$$(-1)^n\sum_{k,\ell}(-1)^{\epsilon_4}c^{-1}_{i_kj_\ell}M^\vee \otimes u_{i_{k+1}}\dots u_{i_{r_1+1}}u_{i_1}... u_{i_{k-1}}u_{j_{\ell+1}}\dots u_{j_{r_2+1}}u_{j_1}\dots u_{j_{\ell-1}},$$ where $\epsilon_1=(n-1)r(k-1)$,  $\epsilon_2=(n-1)r(\ell-1)$, $\epsilon_3= (n-1)r(k+\ell-2)$ and $\epsilon_4=(n-1)r(k+\ell)$

\end{proof}

\begin{example}
Suppose the intersection form $C$ is given by the identity matrix (and $n$ is even). Then we have $\Delta(u_1u_2u_3)\Delta(u_1u_3u_2)=$ $$(x_1\otimes u_2u_3+x_2\otimes u_3u_1+x_3\otimes u_1u_2)(x_1\otimes u_3u_2+x_3\otimes u_2u_1+x_2\otimes u_1u_3)$$ $$=M^\vee\otimes (u_2u_3u_3u_2+u_3u_1u_1u_3+u_1u_2u_2u_1).$$ \end{example}

\begin{remark}
The conceptual way of understanding the formula is as follows (ignoring signs for a moment). Given two cyclic words (describing elements of $H_{1,r_1})$ and $H_{1,r_2}$), their product is a sum of cyclic words (as elements of $H_{2,r_1+r_2}$) given as a sum over all ways of choosing one letter from each word, deleting the letters and gluing the words together at the incisions with a weight depending on the deleted letters. This gives a combinatorial description of the product if we describe both invariants and coinvariants as cyclic words. 
\end{remark}


To finish off the description of the multiplication we need to consider multiplication with the element $\kappa$ since it is not in the image of $\Delta.$




\begin{lemma}
We have the following formulas for the multiplication.
$$\kappa^2=0$$ and
$$\kappa\Delta(M^\vee\otimes u_{i_1}\dots u_{i_r})=(-1)^nrM^\vee\otimes u_{i_1}\dots u_{i_r}.$$

\end{lemma}

\begin{proof}
The first formula follows since the square of an odd element is always zero by graded commutativity. The second follows by the following calculation. 
$$\sum_i (x_i\otimes u_i)(\sum_{k,\ell}(-1)^{(n-1)r(k-1)}c^{-1}_{i_k\ell}x_\ell\otimes u_{i_{k+1}}\dots u_{i_{k-1}})=$$ $$\sum_{i,k,\ell}(-1)^{(n-1)r(k-1)}c_{i\ell}c^{-1}_{i_k\ell}M^\vee\otimes u_iu_{i_{k+1}}\dots u_{i_{k-1}}=$$ $$\sum_{i,k}(-1)^{(n-1)r(k-1)+n}\delta_{ii_k}M^\vee\otimes u_iu_{i_{k+1}}\dots u_{i_{k-1}}=$$ $$\sum_{k}(-1)^{(n-1)r(k-1)+n}M^\vee\otimes u_ku_{i_{k+1}}\dots u_{i_{k-1}}=$$ $$(-1)^nrM^\vee\otimes u_{i_1}\dots u_{i_r}.$$
\end{proof}

Using this description of the product we can describe the Gerstenhaber bracket easily.

\begin{proposition} Let $a\in H_{1,r_1}$ and $b\in H_{1,r_2}.$ Then $$[a,b]=\Delta(ab).$$
Let $c\in H_{2,r_3}$ and $d\in H_{1,r_4}.$ Then $$[c,d]=-\Delta(c)d.$$
Let $e\in H_{i,r}$ and $f\in H_{0,0}.$ Then $[e,f]=0.$
\end{proposition}

\begin{proof}
All these identities follows easily by applying the identity $$[x,y]=\Delta(xy)-\Delta(x)y-(-1)^{deg(x)}x\Delta(y)$$ and noticing that most terms vanish because of degree reasons.
\end{proof}

\section{Dimension Counting}
From the Hilbert series of Lemma \ref{hilbertserieslemma} and \ref{dimensionlemma} we can calculate the dimensions of the parts $U(r)$ and $R(q).$ These give us Lemma \ref{hilbertseriesdifferencelemma}, which tells us that some groups have the same dimensions. However, to analyze what these dimensions are we have to do some combinatorics, which we are going to do in this section. Since these calculations will rely on Lemma \ref{snakelemmalemma} the whole section assumes characteristic zero.
\begin{lemma}
The dimension of $U(r)$ is $$\sum_{a,b\geq 0\atop a+2b=r}{{a+b}\choose{a}} n^a(-1)^b$$ and the dimension of $R(q)$ is $$\sum_{a,b,c\geq 0\atop a+2b+3c=q-2}{{a+b+c}\choose{a,b,c}}2^an^a(-1)^b(n^2+1)^bn^c,$$ where we are using multinomial coefficients.
\end{lemma}

\begin{proof} Expand the Hilbert series to get the coefficients.
\end{proof}

The first part following lemma is quite standard, but we provide a proof since we are going to prove a variant later. We follow the approach of \cite{Reutenauer}, but without using generating functions.

\begin{lemma}
\label{necklacelemma}
The number of necklaces of length $r$ with letters from an alphabet of $m$ letters are $$\sum_{e\mid r}\frac{1}{e}\sum_{d\mid e}\mu(d)m^{e/d}=\frac{1}{r}\sum_{f\mid r} \phi(f)m^{r/f},$$ where $\mu$ is the M\"obius function and $\phi$ is Euler's totient function. Similarly, the number of necklaces with even period length are $$\sum_{e\mid r\atop 2\mid e}\frac{1}{e}\sum_{d\mid e}\mu(d)m^{e/d}.$$
\end{lemma}

\begin{proof}
First we want to count the number of primitive necklaces $Prim(r)$, that is, necklaces that are not repeating. Necklaces with a fixed period length is in bijection with primitive necklaces of total length the same as the period length. The total number of words of length $r$ from an alphabet of $m$ letters is $m^r.$ This gives the formula $$m^r=\sum_{d\mid r}dPrim(d).$$ Now M\"obius inversion gives $$Prim(r)=1/r\sum_{d\mid r}\mu(d)m^{r/d}.$$ Now, to get the number of necklaces we can just sum over all period lengths to obtain the formulas $\sum_{e\mid r}1/e\sum_{d\mid e}\mu(d)m^{e/d}$ and $\sum_{e\mid r\atop 2\mid e}1/e\sum_{d\mid e}\mu(d)m^{e/d}.$ To get the expression $1/r\sum_{f\mid r} \phi(f)m^{r/f},$ we need to do some simplification of \\ $\sum_{e\mid r}1/e\sum_{d\mid e}\mu(d)m^{e/d}.$ We have $$\sum_{e\mid r}1/e\sum_{d\mid e}\mu(d)m^{e/d}=\sum_{e\mid r}Prim(e)=\sum_{e\mid r}Prim(r/e)=$$ $$\sum_{e\mid r}e/r\sum_{d\mid r/e}\mu(d)m^{r/de}=1/r\sum_{e\mid r}\sum_{d\mid r/e}e\mu(d)m^{r/de}.$$ Now let $de=f$ 
and notice that $d\mid r/e\Leftrightarrow f\mid r.$ By then summing in the other order we obtain that $$1/r\sum_{e\mid r}\sum_{d\mid r/e}e\mu(d)m^{r/de}=1/r\sum_{f\mid r}\sum_{d\mid f}e\mu(d)m^{r/f}=$$ $$1/r\sum_{f\mid r}m^{r/f}\sum_{d\mid f}f/d\mu(d)=1/r\sum_{f\mid r} \phi(f)m^{r/f},$$ where the last equality comes from the classical identity $\sum_{d\mid f}f/d\mu(d)=\phi(f)$ which can be gotten by applying M\"obius inversion to $f=\sum_{d\mid f}\phi(d).$
\end{proof}

\begin{corollary}
\label{Ydimensioncorollary}
Let the notation be as in Lemma \ref{Adlemma}. The dimension $ker(Ad)=coker(Ad)$ is $$\sum_{e\mid r}\frac{1}{e}\sum_{d\mid e}\mu(d)m^{e/d}=\sum_{e\mid r}\frac{1}{e}\sum_{d\mid e}\mu(e/d)m^{d}=\frac{1}{r}\sum_{f\mid r} \phi(f)m^{r/f}=\frac{1}{r}\sum_{f\mid r} \phi(r/f)m^{f},$$ if $n$ or $r$ is odd. If $n$ and $r$ are even the dimension is $$\sum_{e\mid r\atop 2\mid e}\frac{1}{e}\sum_{d\mid e}\mu(d)m^{e/d}=\sum_{e\mid r\atop 2\mid e}\frac{1}{e}\sum_{d\mid e}\mu(e/d)m^{d}.$$
\end{corollary}

\begin{proof}
This follows from Lemma \ref{kercokerlemma} and Lemma \ref{kernelnecklacelemma}. The alternate formulas hold by observing that we still sum over the same terms, just in another order.
\end{proof}

\begin{lemma}
\label{chicaplemma}
Let $\omega\in V^{\otimes2}$ be the relation in $U.$ We have $\omega\otimes V\cap V\otimes \omega=0.$
\end{lemma}

\begin{proof}
The Hilbert series of $R$ is $$\frac{t^2}{1-(2nt-(n^2+1)t^2+nt^3)}=$$ $$t^2(1+(2nt-(n^2+1)t^2+nt^3)+(2nt-(n^2+1)t^2+nt^3)^2+\dots)=t^2+2nt^3+O(t^4),$$ showing that $\dim(\omega\otimes V+ V\otimes \omega)=2n.$ But $\dim(\omega\otimes V\oplus V\otimes \omega)=2n$ as well, so the only possibility is that $\dim(\omega\otimes V\cap V\otimes \omega)=0.$

\end{proof}

\begin{lemma}
\label{12necklacelemma}
Let $W_{12}(d)$ be the number of words of length $d$ containing at least one $"1"$ directly in front of a $"2"$ and let $W_{\neg 12}(d)$ be the number of words of length $d$ containing no $"1"$ directly on front of a $"2".$
Let $R_c(d)$ be the number of words containing "12" or beginning in "2" and ending in "1". Then $$R_c(d)=W_{12}(d)+W_{\neg12}(d-2).$$ The number of necklaces of length $r\ge3$ in an alphabet $\{1,2,\dots,m\}$ containing at least one $"1"$ directly in front of one $2$ is $$\sum_{e\mid r}\frac{1}{e}\sum_{d\mid e}\mu(d)R_c(e/d)=\sum_{e\mid r}\frac{1}{e}\sum_{d\mid e}\mu(e/d)R_c(d)=$$ $$\frac{1}{r}\sum_{f\mid r} \phi(f)R_c(r/f)=\frac{1}{r}\sum_{f\mid r} \phi(r/f)R_c(f).$$ The number of such necklaces that have even period length are $$\sum_{e\mid r\atop 2\mid e}\frac{1}{e}\sum_{d\mid e}\mu(d)R_c(e/d)=\sum_{e\mid r\atop 2\mid e}\frac{1}{e}\sum_{d\mid e}\mu(e/d)R_c(d).$$
\end{lemma}

\begin{proof}
First we want to establish that $$R_c(d)=W_{12}(d)+W_{\neg12}(d-2).$$ The number of words that begin in "2" and end in "1" but does not contain "12" is $W_{\neg12}(d-2).$ Thus $R_c(d)=W_{12}(d)+W_{\neg12}(d-2)$ counts the number of words containing $12$ "cyclically". To establish the formulas we can now follow the proof of Lemma \ref{necklacelemma} mutatis mutandis.
\end{proof}

\begin{corollary}
\label{Xdimensioncorollary}
Let the notation be as in Lemma \ref{snakelemmalemma}  and \ref{12necklacelemma} and let $d\geq 3$. Then $ W_{12}(d)=\dim(R(d))$ and $W_{\neg12}(d)=\dim(U(d)).$ The dimension of $X$ is $$\sum_{e\mid r}\frac{1}{e}\sum_{d\mid e}\mu(d)R_c(e/d)=\sum_{e\mid r}\frac{1}{e}\sum_{d\mid e}\mu(e/d)R_c(d)=$$ $$\frac{1}{r}\sum_{f\mid r} \phi(f)R_c(r/f)=\frac{1}{r}\sum_{f\mid r} \phi(r/f)R_c(f),$$  if $n$ or $r$ is odd. If $n$ and $r$ are even, the dimension is $$\sum_{e\mid r\atop 2\mid e}\frac{1}{e}\sum_{d\mid e}\mu(d)R_c(e/d)=\sum_{e\mid r\atop 2\mid e}\frac{1}{e}\sum_{d\mid e}\mu(e/d)R_c(d).$$
\end{corollary}

\begin{proof}
We can pick a basis of $V^{\otimes d}$ where some of the basis elements are of the type $V^i\otimes \omega \otimes V^{d-i-2}.$ Biject these to the set of words of length $d$ such that these basis elements go to words with $12$ in the corresponding place. This is possible since by Lemma \ref{chicaplemma} we know that $\omega\otimes V\cap V\otimes \omega=0.$ 
By counting these special basis elements of their complement we establish that $ W_{12}(d)=\dim(R(d))$ and $W_{\neg12}(d)=\dim(U(d)).$ Similarly (by also making sure treating words starting in $"2"$ and ending in $"1"$ specially), we establish that $R_c(d)$ is the dimension of a vector space $R(d)+(\omega_{(2)}V^{d-2}\omega_{(1)})\subset V^{\otimes d}.$ 
The choice of basis and bijection can be done so that it is equivariant with respect to the cyclic group. This gives us a bijection of a basis and necklaces with restrictions. 

\end{proof}

\begin{theorem} Suppose that $r\geq 3.$ Let $S_c(d)=m^d-R_c(d)=\dim(V^{\otimes d})-(\dim(R(d))+\dim(U(d-2))).$ The dimensions of the homology groups $H_{1,r-1}$ and $H_{2,r}$ are $$\sum_{e\mid r}\frac{1}{e}\sum_{d\mid e}\mu(d)S_c(e/d)=\sum_{e\mid r}\frac{1}{e}\sum_{d\mid e}\mu(e/d)S_c(d)=$$ $$\frac{1}{r}\sum_{f\mid r} \phi(f)S_c(r/f)=\frac{1}{r}\sum_{f\mid r} \phi(r/f)S_c(f),$$ if $n$ or $r$ is odd. If $n$ and $r$ are even, the dimension is $$\sum_{e\mid r\atop 2\mid e}\frac{1}{e}\sum_{d\mid e}\mu(d)S_c(e/d)=\sum_{e\mid r\atop 2\mid e}\frac{1}{e}\sum_{d\mid e}\mu(e/d)S_c(d).$$
\end{theorem}

\begin{proof}
By Lemma \ref{snakelemmalemma} we know that the dimension is the dimension of $Y$ minus the dimension of $X.$ We have counted the dimensions of $X$ and $Y$ in Corollaries \ref{Ydimensioncorollary} and \ref{Xdimensioncorollary}.
\end{proof}

\begin{corollary}
\label{exponentialgrowthcorollary}
The sequence $\dim(H_{2,2k+1})$ grows at least exponentially.
\end{corollary}

\begin{proof}
The dimension of $H_{2,r}$ when $r$ is odd is given by $\frac{1}{r}\sum_{f\mid r} \phi(f)S_c(r/f).$ We have $$\frac{1}{r}\sum_{f\mid r} \phi(f)S_c(r/f)\geq \frac{1}{r}S_c(r).$$ The number of words of length $r$ not containing $12$ cyclically is counted by $S_c(r).$ A subset of these are the words not containing $1$ at all. Thus $S_c(r)\geq (m-1)^r$ and it follows that $$\dim(H_{2,2k+1})\geq \frac{(m-1)^{2k+1}}{2k+1}.$$
\end{proof}

\begin{remark}
The exponential growth can be proven in the setting of $(n-1)$-connected closed manifolds of dimension at most $(3n-2)$ as well. Since we did not derive an explicit formula for the dimension when the generators have mixed degrees, we will have to use a different argument to find enough words. First note that the we always have the vector space $U/[U,U]$ contained in the homology. Pick one variable and consider only words not containing that variable. This gives a subspace isomorphic to $T/[T,T]$ where $T$ is the tensor algebra on the $m-1$ remaining variables. Since we have assumed that $m\ge3,$ there are at least two variables left, which gives us three cases. If we have two even or two odd variables we are back in the case of Lemma \ref{necklacelemma} and one can do the same calculation as in the end of Corollary \ref{exponentialgrowthcorollary}. If we only have one even and one odd variable we can look at tensor words with odd total degree. This gives a subsequence that has exponential growth. Taking the quotient with the cyclic action corresponds to dividing by $r$ (taking into account that we don't get any sign troubles because we are in a situation with an odd number of odd variables) and thus it will still have exponential growth.
\end{remark}

The rate of growth of the Betti numbers of the free loop space $LM$ is interesting because of the connection to the number of closed geodesics on $M$, as explained in e.g. \cite{FelixOpreaTanre}*{Chapter 5}. As a consequence of our calculation, we get an affirmative answer to Gromov's conjecture (see \cite{FelixOpreaTanre}*{Conjecture 5.3}) for the highly connected manifolds considered here. This generalizes \cite{BasuBasu}*{Theorem C(1)} for 4-manifolds.

\begin{corollary}
Let $\kk$ be a field and let $M$ be an $(n-1)$-connected closed manifold of dimension at most $3n-2$ ($n\geq 2$) with $\dim H^*(M;\kk)>4$. For a generic metric on $M$, the number of geometrically distinct closed geodesics of length $\leq T$ grows exponentially in $T$.
\end{corollary}

\section{Appendix: Homological perturbation theory} \label{appendix:hpt}

\begin{definition}
\label{contractiondefinition}
Suppose $(C,d_C)$ and $(D,d_D)$ are chain complexes with maps $f\colon C\rightarrow D,$ $g\colon D\rightarrow C,$ and a map $h\colon C\rightarrow C$ with $|f|=|g|=0\text{ and } |h|=1$, 
$$
\bigSDR{(C,d_C)}{(D,d_D)}{f}{g}{h}.
$$
Suppose moreover that $f$ and $g$ are chain maps and that
$$d_Ch+hd_C=gf-id_C, \quad fg=id_D,$$ and $$fh=0,\quad hh=0,\quad hg=0.$$
Call a diagram like this a \emph{contraction} with data $(C,D,d_C,d_D,f,g,h)$.
\end{definition}

\begin{remark}
It is harmless to assume the extra identities $$fh=0,\quad hh=0,\quad hg=0.$$ Suppose we have data satisfying all the above identities except these. In \cite{LambeStasheff} it is noted that we can always redefine $h$ such that these are satisfied.
\end{remark}

\begin{definition}
Given a complex with differential $d,$ we say that a perturbation of $d$ is a map $t$ of degree $-1$ on the same complex such that $(d+t)^2=0.$
\end{definition}

\begin{theorem}
\label{basicperturbationlemma}
Suppose given a contraction as in Definition \ref{contractiondefinition} and a perturbation $t$ of $d_C$. If $1-ht$ is invertible, then, setting $\Sigma=t(1-ht)^{-1}$, there is a new contraction with data $(C,D,d_C+t,d_D+t',f',g',h')$,
$$
\bigbigSDR{(C,d_C+t)}{(D,d_D+t')}{f'}{g'}{h'},
$$
where
$$
t' = f\Sigma g, \quad f' = f + f\Sigma h \quad g' = g + h\Sigma g,\quad h' = h + h\Sigma h.
$$


%

\end{theorem}

\begin{proof}
See \cite{RBrown}, \cite{Gugenheim} and \cite{BarnesLambe}. 
\end{proof}

\section{Appendix: $A_\infty$-structures and formality} \label{appendix:ainfty}

\begin{definition}
An $A_\infty$-structure on a homologically graded dg vector space $(V,d_V)$ consists of a collection of maps $\{a_n\}_{n\geq 2}$ of degree $n-2$ such that the following identities are satisfied,

$$\sum_{n=i+j+k\atop k\geq 1, i,j \geq 0}(-1)^{i+jk}a_{i+1+j}(id^{\otimes i}\otimes a_k \otimes id^{\otimes j})=0,$$ where we denote $d_V=a_1.$
\end{definition}

\begin{remark}
Note that if $a_n=0$ dor all $n\geq 3$ this is equivalent to the data of a usual dg algebra.
\end{remark}

\begin{proposition}
\label{transfertheorem}
Given a contraction
$$
\bigSDR{(A,d_A)}{(H_*(A),0)}{f}{g}{h}
$$
where $(A,d_A)$ is a dg algebra, there is an $A_\infty$-structure $\{a_n\}_{n\geq 2}$ on the homology $H_*(A)$, where $a_n$ is given by an alternating sum over all rooted trees with $n$ leaves as described in \cite{KontsevichSoibelman} (or in \cite{Berglund2} for a proof using homological perturbation theory). If this $A_\infty$-structure is such that $a_n=0$ for $n\geq 3,$ then $(A,d_A)$ is a formal dg-algebra.
\end{proposition}

\begin{proof}
The $A_\infty$-structure is proven in \cite{KontsevichSoibelman} (or in \cite{Berglund2} for a proof using homological perturbation theory). That it is formal if the higher products vanish follows from \cite{Kadeishvili}.
\end{proof}

\begin{example}
The map $a_3$ is described by the following expression.
$$a_3 = \rtree - \ltree$$
The higher arity case is similar but with more internal edges decorated with $h.$ For our application we will not need to worry about signs, since we are interested in the case when they vanish.
\end{example}




\end{document}